\documentclass[]{article}

\usepackage{floatrow}
\newfloatcommand{capbtabbox}{table}[][\FBwidth]

\usepackage{subfig}

\usepackage{amsfonts}
\usepackage{titling}
\usepackage{amssymb}
\usepackage[margin=1.0in]{geometry}
\usepackage{multicol}
\usepackage[utf8]{inputenc}
\usepackage{graphicx}
\usepackage{setspace}
\usepackage[titletoc,title]{appendix}
\usepackage{lettrine}
\usepackage{changepage}
\usepackage[english]{babel}
\usepackage[section]{placeins}
\usepackage{gensymb}
\usepackage{amsmath,amsthm}
\usepackage{bm}
\usepackage{hyperref}
\usepackage{enumerate}
\usepackage{polynom}
\usepackage{enumitem}
\usepackage{mathtools}
\usepackage[mathscr]{euscript}
\usepackage{tikz}
\usepackage{graphicx}
\usepackage{authblk}
\usepackage[numbers,sort]{natbib}

\usepackage{array}
\newcolumntype{C}[1]{>{\centering\let\newline\\\arraybackslash\hspace{0pt}}m{#1}}

\usepackage{hhline}
\usepackage{tikz}
\usepackage{amssymb,amsthm}
\usepackage[T1]{fontenc}
\usepackage{lmodern}

\usepackage{algorithm}
\usepackage{algpseudocode}

\usepackage{booktabs}

\newcommand{\R}{\mathbb{R}}

\DeclareMathOperator*{\argmin}{arg\,min}

\newcommand{\symnum}[3][2]{\stackrel{\mathclap{\raisebox{.5pt}{\footnotesize \textcircled{\raisebox{-.9pt} {#2}}}}}{#3}}

\newcommand{\numcirc}[1]{$\raisebox{.5pt}{\footnotesize \textcircled{\raisebox{-.9pt} {#1}}}$}

\newtheoremstyle{dotless}{}{}{\itshape}{}{\bfseries}{}{ }{}
\theoremstyle{dotless}
\newtheorem*{proposition*}{Proposition}

\newcommand\defeq{\stackrel{\mathclap{\tiny \mbox{def}}}{=}}

\allowdisplaybreaks

\title{Accelerating Variance-Reduced Stochastic Gradient Methods}

\author[1]{Derek Driggs\thanks{\texttt{d.driggs@damtp.cam.ac.uk}}}
\author[2]{Matthias J. Ehrhardt\thanks{\texttt{m.ehrhardt@bath.ac.uk}}}
\author[1]{Carola-Bibiane Sch\"onlieb\thanks{\texttt{cbs31@cam.ac.uk}}}
\affil[1]{Department of Applied Mathematics and Theoretical Physics, Cambridge University}
\affil[2]{Institute for Mathematical Innovation, University of Bath}

\usetikzlibrary{positioning}

\usepackage{stmaryrd}

\mathtoolsset{showonlyrefs}

\newcommand{\prox}{\textnormal{prox}}
\newcommand{\E}{\mathbb{E}_k}

\newtheorem{theorem}{Theorem}
\newtheorem{lemma}[theorem]{Lemma}
\newtheorem{corollary}[theorem]{Corollary}
\newtheorem{proposition}[theorem]{Proposition}
\newtheorem{remark}{Remark}
\newtheorem{definition}{Definition}

\allowdisplaybreaks

\begin{document}

\maketitle

\begin{abstract}
Variance reduction is a crucial tool for improving the slow convergence of stochastic gradient descent. Only a few variance-reduced methods, however, have yet been shown to directly benefit from Nesterov's acceleration techniques to match the convergence rates of accelerated gradient methods. Such approaches rely on ``negative momentum'', a technique for further variance reduction that is generally specific to the SVRG gradient estimator. In this work, we show that negative momentum is unnecessary for acceleration and develop a universal acceleration framework that allows all popular variance-reduced methods to achieve accelerated convergence rates. The constants appearing in these rates, including their dependence on the number of functions $n$, scale with the mean-squared-error and bias of the gradient estimator. In a series of numerical experiments, we demonstrate that versions of SAGA, SVRG, SARAH, and SARGE using our framework significantly outperform non-accelerated versions and compare favourably with algorithms using negative momentum.
\end{abstract}

% \keywords{Stochastic Optimisation, Convex Optimisation, Variance Reduction, Accelerated Gradient Descent}

% \vspace{2mm}

% \noindent \textbf{MSC} 90C06, 90C15, 90C25, 90C30, 90C60, 68Q25

\section{Introduction}

We are interested in solving the following composite convex minimisation problem:
\begin{equation}
\label{eq:optmain}
    \min_{x \in \R^m} \left\{ F(x) \defeq f(x) + g(x) \defeq \frac{1}{n} \sum_{i=1}^n f_i(x) + g(x) \right\}.
\end{equation}
Throughout, we assume $f_i : \R^m \to \R$ are convex and have $L$-Lipschitz continuous gradients for all $i$. We also assume $g : \R^m \to \R \cup \{\infty\}$ is proper, lower semicontinuous, and $\mu$-strongly convex with $\mu \ge 0$, but we do not require $g$ to be differentiable. Problems of this form are ubiquitous in many fields, including machine learning, compressed sensing, and image processing (see, e.g., \cite{RPCA,matrixcompletion,DonohoMRI,LASSO}). Fundamental examples include LASSO \cite{LASSO} and matrix completion \cite{matrixcompletion}, where $f$ is a least-squares loss and $g$ is the $\ell_1$ or nuclear norm, respectively, and sparse logistic regression, where $f$ is the logistic loss and $g$ is the $\ell_1$ norm.

One well-studied algorithm that solves \eqref{eq:optmain} is the \textit{forward-backward splitting algorithm} \cite{fbs,monotoneOps}. This method has a worst-case convergence rate of $\mathcal{O}\left( 1/T \right)$ when $F$ is not strongly convex, and when $F$ is $\mu$-strongly convex, it converges linearly with a rate of $\mathcal{O}\left( (1 + \kappa^{-1})^{-T} \right)$, where $\kappa \defeq L / \mu$ is the condition number of $F$. The \textit{inertial} forward-backward splitting algorithm \cite{fista} converges at an even faster rate of $\mathcal{O}\left( 1/T^2 \right)$ without strong convexity and a linear rate of $\mathcal{O}\left((1+\kappa^{-1/2})^{-T}\right)$ when $F$ is strongly convex. The inertial forward-backward method is able to achieve these optimal convergence rates because it incorporates \textit{momentum}, using information from previous iterates to adjust the current iterate.

Although the inertial forward-backward algorithm converges quickly, it requires access to the full gradient $\nabla f$ at each iteration, which can be costly, for instance, when $n$ is large. In many applications, common problem sizes are so large that computing $\nabla f$ is prohibitively expensive. Stochastic gradient methods exploit the separable structure of $f$, using the gradient of a few of the components $\nabla f_i$ to estimate the full gradient at the current iterate. In most cases, the complexity of computing $\nabla f_i$ for one $i$ is $1 / n$-times the complexity of computing the full gradient, so stochastic gradient methods generally have a much smaller per-iteration complexity than full-gradient methods. Moreover, it has recently been shown that the optimal convergence rates of stochastic gradient methods are $\mathcal{O}\left( \sqrt{n}/T^2 \right)$ without strong convexity and $\mathcal{O}\left( \theta_S^{-T} \right)$ with $\theta_S \defeq 1 + \sqrt{\frac{\mu}{L n}}$ when $g$ is $\mu$-strongly convex, matching the optimal dependence on $T$ and $\kappa$ of full-gradient methods \cite{srebrocomplexity}.\footnote{The results in \cite{srebrocomplexity} are complexity bounds, bounding the number of gradient and prox oracle calls required to achieve a given tolerance. For algorithms performing $\mathcal{O}(1)$ oracle calls per iteration, these complexity bounds imply the stated bounds on convergence rates.} Stochastic gradient methods have undergone several revolutions to improve their convergence rates before achieving this lower bound. We summarise these revolutions below, beginning with traditional stochastic gradient descent.

\paragraph{Stochastic Gradient Descent (SGD).} \textit{Stochastic gradient descent}, dating back to \cite{sgd}, uses the gradients $\nabla f_j, \ \forall j \in J_k \subset \{1,2,\ldots,n\}$ to estimate the full gradient. The \emph{mini-batch} $J_k$ is an index set chosen uniformly at random from all subsets of $\{ 1,2,\ldots,n \}$ with cardinality $b \defeq |J_k|$. When $b \ll n$, the per-iteration complexity of stochastic gradient descent is much less than full-gradient methods. However, the per-iteration savings come at the cost of a slower convergence rate, as SGD converges at a rate of $\mathcal{O} (1/\sqrt{T})$ in the worst case. Still, SGD outperforms full-gradient methods on many problems, especially if a low-accuracy solution is acceptable.

\paragraph{Variance Reduction.} Variance-reduced estimators use gradient information from previous iterates to construct a better estimate of the gradient at the current step, ensuring that the mean-squared error of these estimates decreases as the iterations increase. Variance-reduction improves the convergence rates of stochastic gradient methods, but either have a higher per-iteration complexity or have larger storage requirements than SGD. The two most popular variance-reduced algorithms are SVRG \cite{svrg} and SAGA \cite{SAGA}, which use the following estimators to approximate $\nabla f(x_{k+1})$:
\begin{align}
\label{eq:svrg}
\widetilde{\nabla}^{\textnormal{\tiny SVRG}}_{k+1} &\defeq \frac{1}{b} \left( \sum_{j \in J_k} \nabla f_j(x_{k+1}) - \nabla f_j(\widetilde{x}) \right) + \nabla f(\widetilde{x}) \\
\label{eq:saga}
\widetilde{\nabla}^{\textnormal{\tiny SAGA}}_{k+1} &\defeq \frac{1}{b} \left( \sum_{j \in J_k} \nabla f_j(x_{k+1}) - \nabla f_j(\varphi_k^j) \right) + \frac{1}{n} \sum_{i=1}^n \nabla f_i(\varphi_k^i).
\end{align}
In SVRG, the full gradient $\nabla f(\widetilde{x})$ is computed every $m \approx 2n$ iterations, and $\nabla f(\widetilde{x})$ is stored and used for future gradient estimators. SAGA takes a similar approach, storing $n$ past stochastic gradients, and updating the stored gradients so that $\nabla f_j(\varphi_{k+1}^j) = \nabla f_j(x_{k+1})$. In this work, we consider a variant of SVRG where the full gradient is computed at every iteration with probability $1/p \in (0,1]$ rather than deterministically computing the full gradient every $2n$ iterations.

SVRG, SAGA, and related variance-reduced methods converge at a rate of $\mathcal{O} \left( n/T \right)$ when no strong convexity is present. With strong convexity, these algorithms enjoy linear convergence, with a rate of $\mathcal{O} ( \left(1 + (n + \kappa)^{-1} \right)^{-T} )$. Although these convergence rates are significantly faster than the rate of SGD, they do not match the asymptotic convergence rates of accelerated first-order methods, converging like $\mathcal{O}\left( n/T^2 \right)$ without strong convexity and $\mathcal{O}( (1+(n \kappa)^{-1/2})^{-T} )$ with strong convexity.

\paragraph{Variance Reduction with Bias.} 
SAGA and SVRG are \textit{unbiased} gradient estimators because they satisfy $\E \widetilde{\nabla}_{k+1} = \nabla f(x_{k+1})$, where $\E$ is the expectation conditioned on the first $k$ iterates. There are several popular variance-reduced algorithms that use biased gradient estimators \cite{SAG,sarah}. In \cite{techrepo}, the authors develop a framework for proving convergence guarantees for biased methods, suggesting that the convergence rates of biased stochastic gradient estimators depend on the sum of two terms:
\begin{equation}
    \gamma^2 \E \|\widetilde{\nabla}_{k+1} - \nabla f(x_{k+1}) \|^2 + \gamma \left\langle \nabla f(x_{k+1}) - \E \widetilde{\nabla}_{k+1}, x_{k+1} - x^* \right\rangle.
\end{equation}
These terms are the mean-squared error (MSE) of the gradient estimator and the ``bias term'', respectively. The authors also show that \textit{recursive} gradient estimators such as SARAH \cite{sarah} and SARGE \cite{techrepo} minimise these terms better than other biased or unbiased estimators, leading to better convergence rates in some settings. The SARAH gradient estimator is
\begin{equation}
\label{eq:sarah}
    \widetilde{\nabla}^{\textnormal{\tiny SARAH}}_{k+1} \defeq \begin{cases}
    \frac{1}{b} \left( \sum_{j \in J_k} \nabla f_j(x_{k+1}) - \nabla f_j(x_k) \right) + \widetilde{\nabla}^{\textnormal{\tiny SARAH}}_k & \textnormal{w.p. } 1 - \frac{1}{p}, \\
    \nabla f(x_{k+1}) & \textnormal{w.p. } \frac{1}{p}.
    \end{cases}
\end{equation}
As with SVRG, we consider a slight variant of the SARAH estimator in this work, where we compute the full gradient at every step with probability $1/p$. The SARGE gradient estimator is similar to the SAGA estimator.
\begin{align}
\label{eq:sarge}
    \widetilde{\nabla}^{\textnormal{\tiny SARGE}}_{k+1} \defeq \frac{1}{b} & \left( \sum_{j \in J_k} \nabla f_j(x_{k+1}) - \psi_k^j \right) + \frac{1}{n} \sum_{i=1}^n \psi_k^i - \left(1 - \frac{b}{n}\right) \left( \frac{1}{b} \sum_{j \in J_k} \nabla f_j(x_k) - \widetilde{\nabla}^{\textnormal{\tiny SARGE}}_k \right),
\end{align}
where the variables $\psi_k^i$ follow the update rule $\psi_{k+1}^j = \nabla f_j(x_{k+1}) - \left(1 - \frac{b}{n}\right)$ $\nabla f_j (x_k)$ for all $j\in J_k$, and $\psi_{k+1}^i = \psi_k^i$ otherwise. Like SAGA, SARGE uses stored gradient information to avoid having to compute the full gradient. These estimators differ from SAGA and SVRG because they are biased (i.e., $\E \widetilde{\nabla}_{k+1} \not = \nabla f(x_{k+1})$). Many works have recently shown that algorithms using the SARAH or SARGE gradient estimators achieve faster convergence rates than algorithms using other estimators in certain settings. Importantly, these recursive gradient methods produce algorithms that achieve the \textit{oracle complexity lower bound} for non-convex composite optimisation \cite{techrepo,spider,spiderm,spiderboost,proxsarah}. They have not yet been shown to achieve optimal convergence rates for convex problems.

\paragraph{Variance Reduction with Negative Momentum.} Starting with Katyusha \cite{katyusha} and followed by many others \cite{Natasha,Natasha2,katyushaX,FSVRG,subsamp,MiG,varag}, a family of stochastic gradient algorithms have recently emerged that achieve the optimal convergence rates implied by \cite{srebrocomplexity}. There are two components to these algorithms that make this acceleration possible. First, these algorithms incorporate momentum into each iteration, either through linear coupling \cite{lincoup}, as in the case of \cite{katyusha,katyushaX,Natasha,Natasha2,subsamp}, or in a more traditional manner reminiscent of Nesterov's accelerated gradient descent \cite{FSVRG,MiG}. Second, these algorithms incorporate an ``anchor-point'' into their momentum updates that supposedly softens the negative effects of bad gradient evaluations. Almost all of these algorithms are an accelerated form of SVRG with updates of the form
\begin{align}
x_{k+1} &= \widetilde{x} + \tau_k(x_k - \widetilde{x}), \quad \textnormal{or} \\
x_{k+1} &= \tau_1 z_k + \tau_2 \widetilde{x} + (1-\tau_1 - \tau_2) y_k,
\end{align}
using traditional acceleration or linear coupling, respectively ($z_k$ and $y_k$ are as defined in Algorithm \ref{alg:1}, and $\tau_k, \tau_1, \tau_2 \in [0,1]$). We see that these updates ``attract'' the current iterate toward a ``safe'' point, $\widetilde{x}$, where we know the full gradient. Because of this ``attractive'' rather than ``repulsive'' quality, updates of this type have been termed ``negative momentum''.

There are several issues with negative momentum. Most importantly, negative momentum is algorithm-specific. Unlike Nesterov's method of momentum or linear coupling, negative momentum cannot be applied to other stochastic gradient algorithms. SAGA, for example, cannot be accelerated using negative momentum of this form, because there does not exist a point $\widetilde{x}$ where we compute the full gradient (however, see \cite{subsamp}). Also, numerical experiments show that negative momentum is often unnecessary to achieve acceleration (see the discussion in \cite{katyusha} or Section \ref{sec:experiments}, for example), suggesting that acceleration is possible without it.

\paragraph{Other Accelerated Methods.} Outside of the family of algorithms using negative-momentum, there exist many stochastic gradient methods that achieve near-optimal convergence rates, including Catalyst \cite{catalyst} and RPDG \cite{rpdg}. Catalyst's convergence rates are a logarithmic factor worse than Katyusha's when the objective is strongly convex or smooth. RPDG achieves optimal convergence rates in the strongly convex setting, matching Katyusha's rate. When strong convexity is not present, RPDG achieves optimal rates up to a logarithmic factor. We include further discussion of these and other related works in Section \ref{sec:prior}.

\paragraph{Contributions.} In this work, we provide a framework to accelerate many stochastic gradient algorithms that does not require negative momentum. We introduce the \emph{MSEB property}, a property that implies natural bounds on the bias and MSE of a gradient estimator, and we prove accelerated convergence rates for all MSEB gradient estimators. As special cases, we show that incorporating the SAGA, SVRG, SARAH, and SARGE gradient estimators into the framework of Algorithm \ref{alg:1} creates a stochastic gradient method with an $\mathcal{O}\left( 1 / T^2 \right)$ convergence rate without strong convexity, and a linear convergence rate that scales with $\sqrt{\kappa}$ when strong convexity is present, achieving the optimal convergence rates in both cases up to a factor of $n$ depending on the bias and MSE of the estimator.

\paragraph{Roadmap.} We introduce our algorithm and state our main result in Section \ref{sec:algs}. We compare our results to existing work in Section \ref{sec:prior}. The next four sections are devoted to proving our main results. In Section \ref{sec:pre}, we review elementary results on the subdifferential relation, results on the proximal operator, and lemmas from convex analysis. We prove a general inequality for accelerated stochastic gradient methods using any stochastic gradient estimator in Section \ref{sec:ineq}. This inequality implies that many stochastic gradient methods can be accelerated using our momentum scheme; to prove an accelerated convergence rate for a specific algorithm, we only need to apply an algorithm-specific bound on the MSE and bias of the gradient estimator. We do this for the SAGA, SVRG, SARAH, and SARGE gradient estimators in Section \ref{sec:convrates}. Finally, in Section \ref{sec:experiments}, we demonstrate the performance of our algorithms in numerical experiments.

\begin{figure}[t]
\begin{center}
\begin{minipage}{0.975\linewidth}
\begin{algorithm}[H]
\caption{A Universal Framework for Acceleration}
\label{alg:1}
  \begin{algorithmic}[1]
  \Require Set step size $\gamma_k$ and momentum parameter $\tau_k$ as in Theorem \ref{thm:main1} if $\mu = 0$ or as in Theorem \ref{thm:main2} otherwise, and gradient estimator $\widetilde{\nabla}$.
  \State Initialise $z_0 = y_0 = x_0$.
  \For{$k = 0,1,\cdots, T-1$}
    \State $x_{k+1} \leftarrow \tau_k z_k + (1-\tau_k)y_k$.
    \State Compute $\widetilde{\nabla}_{k+1}$, an estimate of $\nabla f(x_{k+1})$.
    \State $z_{k+1} \leftarrow \prox_{\gamma_k g}\left( z_k - \gamma_k \widetilde{\nabla}_{k+1} \right)$.
    \State $y_{k+1} \leftarrow \tau_k z_{k+1} + (1-\tau_k) y_k $.
  \EndFor
  \end{algorithmic}
\end{algorithm}
\end{minipage}
\end{center}
\end{figure}

\section{Algorithm and Main Results}
\label{sec:algs}

The algorithm we propose is outlined in Algorithm \ref{alg:1}. Algorithm \ref{alg:1} takes as input any stochastic gradient estimator $\widetilde{\nabla}_{k+1}$, so it can be interpreted as a framework for accelerating existing stochastic gradient methods. This algorithm incorporates momentum through linear coupling \cite{lincoup}, but is related to Nesterov's accelerated gradient method after rewriting $x_{k+1}$ as follows:
\begin{equation}
    x_{k+1} = y_k + (1-\tau_k) (y_k - y_{k-1}).
\end{equation}
With $\tau_k = 1$, there is no momentum, and the momentum becomes more aggressive for smaller $\tau_k$. Although linear coupling provides the impetus for our acceleration framework, similar acceleration schemes appear in earlier works, including Auslender and Teboulle, 2006 \cite{auslender_teboulle_2006}, and Ghadimi and Lan, 2016 \cite{ghadimi_lan_16}.

We show that as long as the MSE and bias of a stochastic gradient estimator satisfy certain bounds and the parameters $\gamma_k$ and $\tau_k$ are chosen correctly, Algorithm \ref{alg:1} converges at an accelerated rate. There are three principles for choosing $\gamma_k$ and $\tau_k$ so that Algorithm \ref{alg:1} achieves acceleration.
\begin{enumerate}
    \item The step size $\gamma_k$ should be small, roughly $\mathcal{O}\left( 1 / n \right)$ with the exact dependence on $n$ decreasing with larger MSE and bias of the gradient estimator.
    \item On non-strongly convex objectives, the step size should grow sufficiently slowly, so that $\gamma_k^2 \left(1 - \rho \right) \le \gamma_{k-1}^2 \left(1 - \frac{\rho}{2}\right)$ with $\rho = \mathcal{O} \left( 1 / n \right)$ decreasing with larger MSE and bias.
    \item The momentum should become more aggressive with smaller step sizes, with $\tau_k = \mathcal{O}\left(\frac{1}{n \gamma_k}\right)$.
\end{enumerate}
For strongly convex objectives, $\gamma_k$ and $\tau_k$ can be kept constant.

For Algorithm \ref{alg:1} to converge, the stochastic gradient estimator must have controlled bias and MSE. Specifically, we require the estimator to satisfy the MSEB property,\footnote{Because this property asserts bounds on the mean-squared-error and bias of a stochastic gradient estimator, the name MSEB is a natural choice. We suggest the pronunciation ``M-SEB''.} introduced below.

\begin{definition}
    For any sequence $\{x_{k+1}\}$, let $\widetilde{\nabla}_{k+1}$ be a stochastic gradient estimator generated from the points $\{x_{\ell+1}\}_{\ell = 0}^k$. The estimator $\widetilde{\nabla}_{k+1}$ satisfies the \emph{MSEB}$(M_1,M_2,\rho_M,$ $\rho_B, \rho_F)$ property if there exist constants $M_1, M_2 \ge 0$, $\rho_M,\rho_B,\rho_F \in (0,1]$, and sequences $\mathcal{M}_k$ and $\mathcal{F}_k$ satisfying
    \begin{equation}
        \nabla f(x_{k+1}) - \E \widetilde{\nabla}_{k+1} = \left( 1 - \rho_B \right) \left( \nabla f(x_k) - \widetilde{\nabla}_k \right),
    \end{equation}
    \begin{align}
        & \mathbb{E} \|\widetilde{\nabla}_{k+1} - \nabla f(x_{k+1})\|^2 \le \mathcal{M}_k,
    \end{align}
    \begin{equation}
    \label{eq:mseb}
        \mathcal{M}_k \le \frac{M_1}{n} \sum_{i=1}^n \mathbb{E} \|\nabla f_i (x_{k+1}) - \nabla f_i(x_k)\|^2 + \mathcal{F}_k + (1-\rho_M) \mathcal{M}_{k-1},
    \end{equation}
    and
    \begin{equation}
        \mathcal{F}_k \le \sum_{\ell=0}^k \frac{M_2 (1-\rho_F)^{k - \ell} }{n} \sum_{i=1}^n \mathbb{E} \|\nabla f_i (x_{\ell+1}) - \nabla f_i(x_\ell)\|^2.
    \end{equation}
\end{definition}
On a high-level, the MSEB property guarantees that the bias and MSE of the gradient estimator decrease sufficiently quickly with $k$.

\begin{remark}
In \cite{nocedalReview}, the authors study the convergence of unbiased stochastic gradient methods under first- and second-moment bounds on the gradient estimator. The bounds implied by the MSEB property are similar, but with the crucial difference that they are \emph{non-Markovian}; we allow our bound on $\mathcal{M}_k$ to depend on all preceding iterates, not just $x_k$.
\end{remark}

In this work, we show that most existing stochastic gradient estimators satisfy the MSEB property, including SAGA, SVRG, SARAH, SARGE, and the full gradient estimator. We list their associated parameters in the following propositions.

\begin{proposition}
\label{prop:full}
    The full gradient estimator $\widetilde{\nabla}_{k+1} = \nabla f(x_{k+1})$ satisfies the MSEB property with $M_1 = M_2 = 0$ and $\rho_M = \rho_B = \rho_F = 1$.
\end{proposition}

\begin{proof}
    The bias and MSE of the full gradient estimator are zero, so it is clear these parameter choices satisfy the bounds in the MSEB property.
\end{proof}

Although trivial, Proposition \ref{prop:full} allows us to show that our analysis recovers the accelerated convergence rates of the inertial forward-backward algorithm as a special case. The MSEB property applies to the SAGA and SVRG estimators non-trivially.

\begin{proposition}
\label{prop:saga}
    The SAGA gradient estimator \eqref{eq:saga} satisfies the MSEB property with $M_1 = \mathcal{O} ( n/b^2 )$, $\rho_M = \mathcal{O} ( b/n )$, $M_2 = 0$, and $\rho_B = \rho_F = 1$. Setting $p = \mathcal{O}( n / b )$, the SVRG gradient estimator \eqref{eq:svrg} satisfies the MSEB property with the same parameters.
\end{proposition}

We prove Proposition \ref{prop:saga} in Appendix \ref{sec:saga}. We are able to choose $\rho_B = 1$ for the SAGA and SVRG gradient estimators because they are unbiased, and we can choose $M_2 = 0$ and $\rho_F = 1$ for these estimators because they admit Markovian bounds on their variance. This is not true for SARAH and SARGE, but these estimators are still compatible with our framework. We prove Propositions \ref{prop:sarah} and \ref{prop:sarge} in Appendices \ref{sec:sarah} and \ref{sec:sarge}, respectively. 

\begin{proposition}
\label{prop:sarah}
    Setting $p = \mathcal{O}( n )$, the SARAH gradient estimator \eqref{eq:sarah} satisfies the MSEB property with $M_1 = \mathcal{O}(1)$, $M_2 = 0$, $\rho_M = \mathcal{O}( 1 / n )$, $\rho_B = \mathcal{O}(1/n)$, and $\rho_F = 1$.
\end{proposition}

\begin{proposition}
\label{prop:sarge}
        The SARGE gradient estimator \eqref{eq:sarge} satisfies the MSEB property with $M_1 = \mathcal{O}(1/n)$, $M_2 = \mathcal{O}( 1/n^2 )$, $\rho_M = \mathcal{O}(b/n)$, $\rho_B = \mathcal{O}(b/n)$, and $\rho_F = \mathcal{O} (b/n)$.
\end{proposition}

All gradient estimators satisfying the MSEB property can be accelerated using the framework of Algorithm \ref{alg:1}, as the following two theorems guarantee.

\begin{theorem}[\textbf{Acceleration Without Strong Convexity}]
\label{thm:main1}

\noindent Suppose the stochastic gradient estimator $\widetilde{\nabla}_{k+1}$ satisfies the \emph{MSEB}$(M_1,M_2,$ $\rho_M,\rho_B,\rho_F)$ property. Define the constants 
\begin{equation*}
\Theta_1 \defeq 1 + \frac{8 (1-\rho_B)}{\rho_B^2 \rho_M}, \quad \Theta_2 \defeq \frac{M_1 \rho_F + 2 M_2}{\rho_M \rho_F}, \quad \textnormal{\emph{and}} \quad \rho \defeq \min\{ \rho_M, \rho_B, \rho_F \}.
\end{equation*}
With
\begin{equation*}
c \ge \max \left\{ \frac{2 \left(1 + \sqrt{1 + 8 \Theta_1 \Theta_2 (2 - \rho_M + \rho_B \rho_M)} \right)}{2 - \rho_M + \rho_B \rho_M}, 16 \Theta_1 \Theta_2 \right\},
\end{equation*}
and $\nu \ge \max \left\{0,\frac{2-6 \rho}{\rho} \right\}$,
set $\gamma_k = \frac{k+\nu+4}{2 c L}$ and $\tau_k = \frac{1}{c L \gamma_k}$. After $T$ iterations, Algorithm \ref{alg:1} produces a point $y_T$ satisfying the following bound on its suboptimality:
\begin{equation}
    \mathbb{E} F(y_T) - F(x^*) \le \frac{K_1 (\nu+2)(\nu+4)}{(T + \nu + 3)^2},
\end{equation}
where
\begin{equation}
    K_1 \defeq F(y_0) - F(x^*) + \frac{2 c L}{(\nu+2)(\nu+4)} \|z_0 - x^*\|^2.
\end{equation}
\end{theorem}

\noindent A similar result gives an accelerated linear convergence rate when strong convexity is present.

\begin{theorem}[\textbf{Acceleration With Strong Convexity}]
\label{thm:main2}

\noindent Suppose the stochastic gradient estimator $\widetilde{\nabla}_{k+1}$ satisfies the \emph{MSEB}$(M_1,M_2,$ $\rho_M,\rho_B,\rho_F)$ property and $g$ is $\mu$-strongly convex with $\mu > 0$. With the constants $\Theta_1, \Theta_2, c$, and $\nu$ set as in Theorem \ref{thm:main1}, set $\gamma = \min \{ \frac{1}{\sqrt{\mu c L}}, \frac{\rho}{2 \mu} \}$ and $\tau = \mu \gamma$. After $T$ iterations, Algorithm \ref{alg:1} produces a point $z_T$ satisfying the following bound:
\begin{equation}
    \mathbb{E} \|z_T - x^*\|^2 \le K_2 \left(1+ \min \left\{ \sqrt{\frac{\mu}{L c}}, \frac{\rho}{2} \right\} \right)^{-T},
\end{equation}
where
\begin{equation}
    K_2 \defeq \frac{2}{\mu} \left(F(y_0) - F(x^*)\right) + \|z_0 - x^*\|^2.
\end{equation}
\end{theorem}

\begin{remark}
    Although we prove accelerated convergence rates for many popular gradient estimators, the generality of Theorems \ref{thm:main1} and \ref{thm:main2} allows our results to extend easily to gradient estimators not considered in this work as well. These include, for example, the gradient estimators considered in \cite{neighbors}.
\end{remark}

\begin{remark}
With some manipulation, we see that these rates with $c = \nu = \mathcal{O}(n)$ are similar to the rates proved for Katyusha. In \cite{katyusha}, the author shows that in the non-strongly convex case, Katyusha satisfies
\begin{equation}
    \mathbb{E} F(\widetilde{x}_S) - F(x^*) \le \mathcal{O} \left( \frac{F(x_0) - F(x^*)}{S^2} + \frac{L \|x_0 - x^*\|^2}{P S^2} \right).
\end{equation}
Recall that Katyusha follows the algorithmic framework of SVRG; $S$ denotes the epoch number, $\widetilde{x}_S$ the point where the full gradient was computed at the beginning of epoch $S$, and $P = \mathcal{O}(n)$ is the epoch length. In our notation, $S = T/P = \mathcal{O}\left( T/n \right)$. Theorem \ref{thm:main1} with $c = \nu = \mathcal{O}(n)$ shows that Algorithm \ref{alg:1} achieves a similar convergence rate of
\begin{equation}
    \mathbb{E} F(y_T) - F(x^*) \le \mathcal{O} \left( \frac{n^2}{T^2} \left( F(y_0) - F(x^*) + \frac{L}{n} \|z_0 - x^*\|^2 \right) \right).
\end{equation}
In the strongly convex case, an appropriately adapted version of Katyusha satisfies
\begin{equation}
    \mathbb{E} F(\widetilde{x}_S) - F(x^*) \le \begin{cases}
    \mathcal{O} \left( \left( 1 + \frac{\sqrt{\mu}}{\sqrt{L P}} \right)^{-S P} \right) & \frac{4}{3} \le \frac{\sqrt{L}}{\sqrt{\mu n}} \\
    \mathcal{O} \left( \left( \frac{3}{2} \right)^{-S} \right) & \frac{\sqrt{L}}{\sqrt{\mu n}} < \frac{4}{3}.
    \end{cases}
\end{equation}
Similarly, with $c = \rho = \mathcal{O}( n )$, Theorem \ref{thm:main2} shows that the iterates of Algorithm \ref{alg:1} satisfy
\begin{align}
    \frac{1}{2} \mathbb{E} \|z_T - x^*\|^2 \le \mathcal{O} \left( \left(1+ \min \left\{ \sqrt{\frac{\mu}{L n}}, \frac{1}{n} \right\} \right)^{-T} \right),
\end{align}
which again matches the rate of Katyusha. Of course, not all stochastic gradient estimators satisfy the bounds necessary to set $c = \nu = \rho = \mathcal{O}( n )$, so these optimal rates are conditional on being able to construct an ``optimal estimator''. SAGA, SVRG, SARAH, and SARGE all require $c$ to be slightly larger than $\mathcal{O}(n)$.
\end{remark}

The proofs of Theorems \ref{thm:main1} and \ref{thm:main2} use a linear coupling argument adapted from \cite{lincoup}, but we use a different adaptation than the one in \cite{katyusha} used to prove convergence rates for Katyusha. To explain the differences between our approach and existing approaches, let us give a high-level description of linear coupling and the generalisation used in \cite{katyusha}.

In \cite{lincoup}, the authors suggest that gradient descent and mirror descent can be coupled to create an accelerated algorithm. We do not discuss gradient descent and mirror descent in detail (for this, see \cite{lincoup}), but the main idea of linear coupling can be understood from only two bounds arising from these algorithms. For the purpose of this argument, suppose $g \equiv 0$, so that $F \equiv f$. Gradient descent with step size $\eta$ satisfies the following bound on the decrease of the objective (equation (2.1) in \cite{lincoup}):
\begin{equation}
\label{eq:gradineq}
    f(x_{k+1}) \le f(x_k) - \frac{1}{\eta} \|\nabla f(x_{k+1})\|^2.
\end{equation}
This bound shows that gradient descent is indeed a \textit{descent method}; it is guaranteed to make progress at each iteration. The iterates of mirror descent using step size $\gamma$ satisfy a bound on the sub-optimality of each iterate (equation (2.2) in \cite{lincoup}).
\begin{equation}
\label{eq:mirrorineq}
    \langle \nabla f(x_k), x_k - x^* \rangle \le \frac{1}{2} \|x_k - x^*\|^2 - \frac{1}{2} \|x_{k+1} - x^*\|^2 + \frac{\gamma^2}{2} \|\nabla f(x_k)\|^2.
\end{equation}
While gradient descent is guaranteed to make progress proportional to $\|\nabla f(x_k)\|^2$ each iteration, mirror descent potentially introduces an ``error'' that is proportional to $\|\nabla f(x_k)\|^2$. Linear coupling takes advantage of this duality. Loosely speaking, by combining the sequence of iterates produced by gradient descent 
with the sequence produced by mirror descent, the guaranteed progress of gradient descent balances the potential error introduced by mirror descent, accelerating convergence.

This argument does not immediately hold for stochastic gradient methods. This is because in addition to the norm $\|\nabla f(x_k)\|^2$ arising in inequalities \eqref{eq:gradineq} and \eqref{eq:mirrorineq}, we also get the MSE of our gradient estimator $\|\widetilde{\nabla}_k - \nabla f(x_k)\|^2$ as well as a ``bias term''. In the stochastic setting, analogues of inequalities \eqref{eq:gradineq} and \eqref{eq:mirrorineq} read
\begin{align}
\label{eq:grad2}
     f(x_{k+1}) & \le f(x_k) + \langle \nabla f(x_{k+1}),  x_{k+1} - x_k \rangle \\
     & = f(x_k) - \frac{1}{\eta} \|\widetilde{\nabla}_{k+1}\|^2 + \langle \nabla f(x_{k+1}) - \widetilde{\nabla}_{k+1}, x_{k+1} - x_k \rangle \\
     & \le f(x_k) + \left(\frac{\epsilon}{2 \eta^2} - \frac{1}{\eta} \right) \|\widetilde{\nabla}_{k+1}\|^2 + \frac{1}{2 \epsilon} \| \widetilde{\nabla}_{k+1} - \nabla f(x_{k+1})\|^2,
\end{align}
where the last inequality is Young's, and
\begin{align}
\label{eq:mirror2}
    \gamma ( f(x_k) - f(x^*) ) \le & \frac{1}{2} \|x_k - x^*\|^2 - \frac{1}{2} \|x_{k+1} - x^*\|^2 + \gamma \left\langle \nabla f(x_k) - \widetilde{\nabla}_k, x_k - x^* \right\rangle + \frac{\gamma^2}{2} \|\widetilde{\nabla}_k\|^2.
\end{align}
If the MSE or bias term is too large, the gradient step is no longer a descent step, and the progress does not balance the ``error terms'' in each of these inequalities, so we cannot expect linear coupling to offer any acceleration. This problem with the MSE and bias term exists for non-accelerated algorithms as well, and all analyses of stochastic gradient methods bound the effect of these terms, but in different ways. Katyusha and other accelerated algorithms in this family incorporate negative momentum to cancel part of the MSE. In contrast, analyses of non-accelerated algorithms do not try to cancel any of the variance, but show that the variance decreases fast enough so that it does not affect convergence rates.

\section{Related Work}
\label{sec:prior}

Besides Katyusha, there are many algorithms that use negative momentum for acceleration. In \cite{FSVRG}, the authors consider an accelerated version of SVRG that combines negative momentum with Nesterov's momentum to achieve the optimal $\sqrt{\kappa}$ dependence in the strongly convex case. This approach to acceleration is almost the same as Katyusha, but uses a traditional form of Nesterov's momentum instead of linear coupling. MiG \cite{MiG} is another variant of these algorithms, corresponding to Katyusha with a certain parameter set to zero. VARAG is another approach to accelerated SVRG using negative momentum. VARAG achieves optimal convergence rates in the non-strongly convex and strongly convex settings under the framework of a single algorithm, and it converges linearly on problems that admit a global error bound, a quality that other algorithms have not yet been shown to possess \cite{varag}. 

The only direct acceleration of a SAGA-like algorithm is SSNM from \cite{subsamp}. Using the notation of \eqref{eq:saga}, SSNM chooses a point from the set $\{\varphi_k^i\}_{i=1}^n$ uniformly at random, and uses this point as the ``anchor point'' for negative-momentum acceleration. Although SSNM admits fast convergence rates, there are a few undesirable qualities of this approach. SAGA has heavy storage requirements because it must store $n$ gradients from previous iterations, and SSNM exacerbates this storage problem by storing $n$ points from previous iterations as well. SSNM must also compute two stochastic gradients each iteration, so its per-iteration computational cost is similar to SVRG and Katyusha, and always higher than SAGA's.

Many algorithms for non-convex optimisation also use negative momentum for acceleration. KatyushaX \cite{katyushaX} is a version of Katyusha adapted to optimise sum-of-non-convex objectives. To achieve its acceleration, KatyushaX uses classical momentum and a ``retraction step'', which is effectively an application of negative momentum (this relationship is acknowledged in \cite{katyushaX} as well). Natasha \cite{Natasha} and Natasha2 \cite{Natasha2} are accelerated algorithms for finding stationary points of non-convex objectives. Both algorithms employ a ``retraction step'' that is similar to negative momentum \cite{Natasha}.

There are also many accelerated stochastic gradient algorithms that do not use negative momentum. In \cite{Nitanda}, the author applies Nesterov's momentum to SVRG without any sort of negative momentum, proving a linear convergence rate in the strongly convex regime. However, the proven convergence rate is suboptimal, as it implies even worse performance than SVRG when the batch size is small and worse performance than accelerated full-gradient methods when the batch size is close to $n$. Our results show that a particular application of Nesterov's momentum to SVRG does provide acceleration. 

Point-SAGA \cite{pointSAGA} is another SAGA-like algorithm that achieves optimal convergence rates, but point-SAGA must compute the proximal operator corresponding to $F$ rather than the proximal operator corresponding to $g$. This is not possible in general, even if the proximal operator corresponding to $g$ is easy to compute, so point-SAGA applies to a different class of functions than the class we consider in this work.

There are also many algorithms that indirectly accelerate stochastic gradient methods. This class of algorithms include Catalyst \cite{catalyst}, APPA \cite{APPA}, and the primal-dual methods in \cite{ZhangXiaoPrimalDual}. These algorithms call a variance-reduced stochastic gradient method as a subroutine, and provide acceleration using an inner-outer loop structure. These algorithms are often difficult to implement in practice due to the difficulty of solving their inner-loop subproblems, and they achieve a convergence rate that is only optimal up to a logarithmic factor.

\section{Preliminaries}
\label{sec:pre}

In this section, we present some basic definitions and results from optimisation and convex analysis. Much of our analysis involves \emph{Bregman divergences}. The Bregman divergence associated with a function $h$ is defined as
\begin{equation}
    D^\xi_h(y,x) \defeq h(y) - h(x) + \langle \xi, x - y \rangle,
\end{equation}
where $\xi \in \partial h(x)$ and $\partial$ is the subdifferential operator. If $h$ is differentiable, we drop the superscript $\xi$ as the subgradient is unique. The function $h$ is convex if and only if $D^\xi_h(y,x) \ge 0$ for all $x$ and $y$. We say $h$ is $\mu$-\emph{strongly convex} with $\mu 
\ge 0$ if and only if
\begin{equation}
\frac{\mu}{2} \|x - y\|^2 \le D^\xi_h(y,x).
\end{equation}
Bregman divergences also arise in the following fundamental inequality.

\begin{lemma}[\cite{nest2004}, Thm. 2.1.5]
\label{lem:2L}
Suppose $f$ is convex with an $L$-Lipschitz continuous gradient. We have for all $x,y \in \R^m$,
\begin{equation}
\|\nabla f(x) - \nabla f(y)\|^2 \le 2 L D_f(y,x).
\end{equation}
\end{lemma}
Lemma \ref{lem:2L} is equivalent to the following result, which is more specific to our analysis due to the finite-sum structure of the smooth term in \eqref{eq:optmain}.

\begin{lemma}
\label{lem:varcan}
Let $f(x) = \tfrac{1}{n} \sum_{i=1}^n f_i(x)$, where each $f_i$ is convex with an $L$-Lipschitz continuous gradient. Then for every $x, y \in \R^m$
\begin{equation}
\frac{1}{n} \sum_{i=1}^n \|\nabla f_i(x) - \nabla f_i(y)\|^2 \le 2 L D_f(y,x).
\end{equation}
\end{lemma}

\begin{proof}
    This follows from applying Lemma \ref{lem:2L} to each component $f_i$.
\end{proof}

The \emph{proximal operator} is defined as
\begin{equation}
\prox_{g} (y) = \argmin_{x \in \R^m} \left\{ \frac{1}{2} \|x-y\|^2 + g(x) \right\}.
\end{equation}
The proximal operator is also defined implicitly as $y - \prox_g(y) \in \partial g(\prox_g$ $(y))$. From this definition of the proximal operator, the following standard inequality is clear.

\begin{lemma}
\label{lem:strongcon}
Suppose $g$ is $\mu$-strongly convex with $\mu \ge 0$, and suppose $z = \prox_{\eta g} \left( x - \eta d \right)$ for some $x, d \in \R^m$ and constant $\eta$. Then, for any $y \in \R^m$,
\begin{equation}
\eta \langle d, z - y \rangle \le \frac{1}{2} \|x - y\|^2 - \frac{1 + \mu \eta}{2} \|z - y\|^2 - \frac{1}{2} \|z - x\|^2 - \eta g(z) + \eta g(y).
\end{equation}
\end{lemma}

\begin{proof}
    By the strong convexity of $g$,
    \begin{equation}
        g(z) - g(y) \le \langle \xi, z - y \rangle - \frac{\mu}{2} \|z - y\|^2 \quad \forall \xi \in \partial g(z).
    \end{equation}
    From the implicit definition of the proximal operator, we know that $\frac{1}{\eta} (z - x) + d \in \partial g(z)$. Therefore,
    \begin{align}
        g(z) - g(y) &\le \langle \xi, z - y \rangle - \frac{\mu}{2} \|z - y\|^2 \\
        & = \frac{1}{\eta} \langle z - x + \eta d, z - y \rangle  - \frac{\mu}{2} \|z - y\|^2 \\
        & = \langle d, z - y \rangle + \frac{1}{2 \eta} \|x - y\|^2 - \frac{1 + \mu \eta}{2 \eta} \|z - y\|^2 - \frac{1}{2 \eta} \|z - x\|^2.
    \end{align}
    Multiplying by $\eta$ and rearranging yields the assertion.
\end{proof}

\section{The Acceleration Framework}
\label{sec:ineq}

To apply the linear coupling framework, we must couple stochastic analogues of 
\eqref{eq:gradineq} and \eqref{eq:mirrorineq} to construct a lower bound on the one-iteration progress of Algorithm \ref{alg:1}.

\begin{lemma}[\textbf{One-Iteration Progress}]
\label{lem:main}
The following bound describes the progress made by one iteration of Algorithm \ref{alg:1}.
\begin{align}
\label{eq:main}
0 \le & \frac{\gamma_k(1-\tau_k)}{\tau_k} F(y_k)
- \frac{\gamma_k}{\tau_k} F(y_{k+1}) + \gamma_k F(x^*) + \gamma_k^2  \|\widetilde{\nabla}_{k+1} - \nabla f(x_{k+1})\|^2 \notag \\
& + \frac{\gamma_k}{\tau_k} \left( \frac{L}{2} - \frac{1}{4 \tau_k \gamma_k} \right) \|x_{k+1} - y_{k+1}\|^2 + \frac{1}{2} \|z_k - x^*\|^2 \notag \\
& - \frac{1 + \mu \gamma_k}{2} \|z_{k+1} - x^*\|^2 + \gamma_k \left\langle \nabla f (x_{k+1}) - \widetilde{\nabla}_{k+1}, z_k - x^* \right\rangle \notag \\
& - \frac{\gamma_k(1-\tau_k)}{\tau_k} D_f(y_k,x_{k+1}).
\end{align}
\end{lemma}

\begin{proof}

We use a linear coupling argument. The extrapolated iterate $x_{k+1}$ can be viewed as a convex combination of an iterate produced from mirror descent (namely, $z_k$) and one from gradient descent ($y_k$). This allows us to provide two bounds on the term $f(x_{k+1}) - f(x^*)$: one is a regret bound inspired by the classical analysis of mirror descent, and the other is inspired by the traditional descent guarantee of gradient descent.
\begin{align}
\label{eq:total}
& \gamma_k (f(x_{k+1}) - f(x^*)) \notag \\
&\symnum{1}{\le} \gamma_k\langle \nabla f(x_{k+1}),x_{k+1} - x^* \rangle \notag \\
%%%%%%%%%%%%%%%%%%%%%%%%%%%%%%%%%%%%%%%%%%%%%
& = \gamma_k \langle \nabla f(x_{k+1}),x_{k+1} - z_k \rangle + \gamma_k \langle \nabla f(x_{k+1}), z_k - x^* \rangle \notag \\
%%%%%%%%%%%%%%%%%%%%%%%%%%%%%%%%%%%%%%%%%%%%%
&\symnum{2}{=} \frac{\gamma_k(1 - \tau_k)}{\tau_k} \langle \nabla f(x_{k+1}), y_k - x_{k+1} \rangle + \gamma_k\langle \nabla f(x_{k+1}), z_k - x^* \rangle \notag \\
%%%%%%%%%%%%%%%%%%%%%%%%%%%%%%%%%%%%%%%%%%%%%
&= \frac{\gamma_k(1-\tau_k)}{\tau_k} ( f(y_k) - f(x_{k+1}) ) + \gamma_k  \left\langle \widetilde{\nabla}_{k+1}, z_k - x^* \right\rangle \notag \\
& \quad \quad - \frac{\gamma_k(1 - \tau_k)}{\tau_k} D_f(y_k,x_{k+1}) + \gamma_k \left\langle \nabla f (x_{k+1}) - \widetilde{\nabla}_{k+1}, z_k - x^* \right\rangle \notag \\
%%%%%%%%%%%%%%%%%%%%%%%%%%%%%%%%%%%%%%%%%%%%%
&= \frac{\gamma_k(1-\tau_k)}{\tau_k} ( f(y_k) - f(x_{k+1}) ) + \gamma_k  \left\langle \widetilde{\nabla}_{k+1}, z_k - z_{k+1} \right\rangle \notag \\
& \quad \quad + \gamma_k  \left\langle \widetilde{\nabla}_{k+1}, z_{k+1} - x^* \right\rangle - \frac{\gamma_k(1 - \tau_k)}{\tau_k} D_f(y_k,x_{k+1}) \notag \\
& \quad \quad + \gamma_k \left\langle \nabla f (x_{k+1}) - \widetilde{\nabla}_{k+1}, z_k - x^* \right\rangle \notag \\
%%%%%%%%%%%%%%%%%%%%%%%%%%%%%%%%%%%%%%%%%%%%%
&\symnum{3}{=} \frac{\gamma_k(1-\tau_k)}{\tau_k} ( f(y_k) - f(x_{k+1}) ) + \frac{\gamma_k}{\tau_k} \left\langle \widetilde{\nabla}_{k+1}, x_{k+1} - y_{k+1} \right\rangle \notag \\
& \quad \quad + \gamma_k  \left\langle \widetilde{\nabla}_{k+1}, z_{k+1} - x^* \right\rangle - \frac{\gamma_k(1 - \tau_k)}{\tau_k} D_f(y_k,x_{k+1}) \notag \\
& \quad \quad + \gamma_k \left\langle \nabla f (x_{k+1}) - \widetilde{\nabla}_{k+1}, z_k - x^* \right\rangle
\end{align}
Inequality \numcirc{1} uses the convexity of $f$, \numcirc{2} follows from the fact that $x_{k+1} = \tau_k z_k + (1 - \tau_k) y_k$, and \numcirc{3} uses $x_{k+1} - y_{k+1} = \tau_k (z_k - z_{k+1})$. We proceed to bound the inner product $\langle \widetilde{\nabla}_{k+1}, z_{k+1} - x^* \rangle$ involving the sequence $z_{k+1}$ using a regret bound from mirror descent, and we bound the term $\langle \widetilde{\nabla}_{k+1}, x_{k+1} - y_{k+1} \rangle$ using an argument similar to the descent guarantee of gradient descent.

By Lemma \ref{lem:strongcon} with $z = z_{k+1}$, $x = z_k$, $y = x^*$, $d = \widetilde{\nabla}_{k+1}$, and $\eta =  \gamma_k$,
\begin{align}
\label{eq:mirror}
    & \gamma_k \left\langle \widetilde{\nabla}_{k+1}, z_{k+1} - x^* \right\rangle \notag \\
    & \le \frac{1}{2 } \|z_k - x^*\|^2 - \frac{1 +  \mu \gamma_k}{2 } \|z_{k+1} - x^*\|^2 - \frac{1}{2 } \|z_{k+1} - z_k\|^2 \notag \\
    & \quad \quad - \gamma_k g(z_{k+1}) + \gamma_k g(x^*) \notag \\
    & = \frac{1}{2 } \|z_k - x^*\|^2 - \frac{1 +  \mu \gamma_k}{2 } \|z_{k+1} - x^*\|^2 - \frac{1}{2  \tau_k^2} \|x_{k+1} - y_{k+1}\|^2 \notag \\
    & \quad \quad  - \gamma_k g(z_{k+1}) + \gamma_k g(x^*).
\end{align}
For the other term,
\begin{align}
\label{eq:grad}
%%%%%%%%%%%%%%%%%%%%%%%%%%%%%%%%%%%%%%%%%%%%%%%%%%%%%%%%%%%%%%%%%%%%%%%%%%%%
& \frac{\gamma_k}{\tau_k} \langle \widetilde{\nabla}_{k+1}, x_{k+1} - y_{k+1} \rangle \notag \\
%%%%%%%%%%%%%%%%%%%%%%%%%%%%%%%%%%%%%%%%%%%%%%%%%%%%%%%%%%%%%%%%%%%%%%%%%
= & \frac{\gamma_k}{\tau_k} \langle \nabla f(x_{k+1}), x_{k+1} - y_{k+1} \rangle + \frac{\gamma_k}{\tau_k} \langle \widetilde{\nabla}_{k+1} - \nabla f(x_{k+1}), x_{k+1} - y_{k+1} \rangle \notag \\
%%%%%%%%%%%%%%%%%%%%%%%%%%%%%%%%%%%%%%%%%%%%%%%%%%%%%%%%%%%%%%%%%%%%%%%%%%
\symnum{1}{\le} & \frac{\gamma_k}{\tau_k} \left( f(x_{k+1}) - f(y_{k+1}) \right) + \frac{\gamma_k}{\tau_k} \langle \widetilde{\nabla}_{k+1} - \nabla f(x_{k+1}), x_{k+1} - y_{k+1} \rangle \notag \\
& + \frac{L \gamma_k}{2 \tau_k} \|x_{k+1} - y_{k+1} \|^2 \notag \\
%%%%%%%%%%%%%%%%%%%%%%%%%%%%%%%%%%%%%%%%%%%%%%%%%%%%%%%%%%%%%%%%%%%%%%%%%%
\symnum{2}{\le} & \frac{\gamma_k}{\tau_k} \left( f(x_{k+1}) - f(y_{k+1}) \right) + \gamma_k^2  \| \widetilde{\nabla}_{k+1} - \nabla f(x_{k+1})\|^2 \notag \\
& + \left(\frac{L \gamma_k}{2 \tau_k} + \frac{1}{4  \tau_k^2} \right) \|x_{k+1} - y_{k+1} \|^2 \notag \\
%%%%%%%%%%%%%%%%%%%%%%%%%%%%%%%%%%%%%%%%%%%%%%%%%%%%%%%%%%%%%%%%%%%%%%%%%%
= & \frac{\gamma_k}{\tau_k} \left( f(x_{k+1}) - F(y_{k+1}) \right) + \gamma_k^2  \| \widetilde{\nabla}_{k+1} - \nabla f(x_{k+1})\|^2 \notag \\
&+ \left(\frac{L \gamma_k}{2 \tau_k} + \frac{1}{4  \tau_k^2} \right)\|x_{k+1} - y_{k+1} \|^2 + \frac{\gamma_k}{\tau_k} g(y_{k+1}) \notag \\
%%%%%%%%%%%%%%%%%%%%%%%%%%%%%%%%%%%%%%%%%%%%%%%%%%%%%%%%%%%%%%%%%%%%%%%%%%
\symnum{3}{\le} & \frac{\gamma_k}{\tau_k} \left( f(x_{k+1}) - F(y_{k+1}) \right) + \gamma_k^2  \| \widetilde{\nabla}_{k+1} - \nabla f(x_{k+1})\|^2 \notag \\
& + \left(\frac{L \gamma_k}{2 \tau_k} + \frac{1}{4 \tau_k^2} \right)\|x_{k+1} - y_{k+1} \|^2 + \gamma_k g(z_{k+1}) + \frac{\gamma_k (1-\tau_k)}{\tau_k} g(y_k).
\end{align}
Inequality $\raisebox{.5pt}{\footnotesize \textcircled{\raisebox{-.9pt} {1}}}$ follows from the Lipschitz continuity of $\nabla f_i$, $\raisebox{.5pt}{\footnotesize \textcircled{\raisebox{-.9pt} {2}}}$ is Young's inequality, and $\raisebox{.5pt}{\footnotesize \textcircled{\raisebox{-.9pt} {3}}}$ uses the convexity of $g$ and the update rule $y_{k+1} = \tau_k z_{k+1} + (1-\tau_k) y_k$. Combining inequalities \eqref{eq:mirror} and \eqref{eq:grad} with \eqref{eq:total} and rearranging yields the assertion.
\end{proof}

Lemma \ref{lem:main} completes the linear coupling part of our argument. If not for the MSE and bias terms, we could telescope this inequality as in \cite{lincoup} and prove an accelerated convergence rate. As with all analyses of stochastic gradient methods, we need a useful bound on these qualities of the estimator.

Existing analyses of unbiased stochastic gradient methods bound the variance term by a pair of terms that telescope over several iterations, showing that the variance tends to zero with the number of iterations. It is difficult to generalise these arguments to accelerated stochastic methods because one must prove that the variance decreases at an accelerated rate that is inconsistent with existing variance bounds. In the analysis of Katyusha, negative momentum cancels part of the variance term, leaving telescoping terms that decrease at an accelerated rate. Without negative momentum, we must handle the variance term differently.

In the inequality of Lemma \ref{lem:main}, we have two non-positive terms:
\begin{equation}
    - \frac{1}{\tau_k^2} \|x_{k+1} - y_{k+1}\|^2 \quad \textnormal{and} \quad - \frac{\gamma_k (1-\tau_k)}{\tau_k} D_f(y_k, x_{k+1}).
\end{equation}
This makes our strategy clear: we must bound the MSE and bias terms by terms of the form $\|x_{k+1} - y_{k+1}\|^2$ and $D_f(y_k,x_{k+1})$. The following two lemmas use the MSEB property to establish bounds of this form.

\begin{lemma}[\textbf{Bias Term Bound}]
\label{lem:biasbound}
    Suppose the stochastic gradient estimator $\widetilde{\nabla}_{k+1}$ satisfies the \emph{MSEB}$(M_1$, $M_2,$ $\rho_M,\rho_B,\rho_F)$ property, let $\rho = \min\{ \rho_M, \rho_B, \rho_F \}$, and let $\{\sigma_k\}$ and $\{s_k\}$ be any non-negative sequences satisfying $\sigma_k s_k^2 \left( 1 - \rho \right) \le \sigma_{k-1} s^2_{k-1} \left( 1 - \frac{\rho}{2} \right)$ and $\sigma_k \left( 1 - \rho \right) \le \sigma_{k-1} \left( 1 - \frac{\rho}{2} \right)$. The bias term can be bounded as
    \begin{align}
    \label{thm1:bias}
        & \sum_{k=0}^{T-1} \sigma_k s_k \mathbb{E} \left\langle \nabla f(x_{k+1}) - \widetilde{\nabla}_{k+1}, z_k - x^* \right\rangle \notag \\
        \le & (1-\rho_B) \sum_{k=0}^{T-1} \sigma_k \mathbb{E} \left[ \frac{8 s^2_k}{\rho_B^2 \rho_M} \left\| \nabla f(x_{k+1}) - \widetilde{\nabla}_{k+1} \right\|^2 + \frac{\rho_M}{8 \tau_k^2} \|x_{k+1} - y_{k+1}\|^2 \right].
    \end{align}
\end{lemma}

\begin{proof}

Because $z_k$ depends only on the first $k-1$ iterates, we can use the MSEB property to say
\begin{align}
    & \sigma_k s_k \mathbb{E} \left\langle \nabla f(x_{k+1}) - \widetilde{\nabla}_{k+1}, z_k - x^* \right\rangle \\
    = & \ \sigma_k s_k \mathbb{E} \left\langle \nabla f(x_{k+1}) - \E \widetilde{\nabla}_{k+1}, z_k - x^* \right\rangle \\
    \symnum{1}{=} & \ \sigma_k s_k (1-\rho_B) \mathbb{E} \left\langle \nabla f(x_k) - \widetilde{\nabla}_k, z_k - x^* \right\rangle \\
    %%%%%%%%%%%%%%%%%%%%%%%%%%%%%%%%%%%%%%%%%%%%%%%%%%%%%%%%%%
    \symnum{2}{=} & \ \sigma_k (1-\rho_B) \mathbb{E} \Big[ s_k \left\langle \nabla f(x_k) - \widetilde{\nabla}_k, z_k - z_{k-1} \right\rangle + s_k \left\langle \nabla f(x_k) - \mathbb{E}_{k-1} \widetilde{\nabla}_k, z_{k-1} - x^* \right\rangle \Big] \\
    %%%%%%%%%%%%%%%%%%%%%%%%%%%%%%%%%%%%%%%%%%%%%%%%%%%%%%%%%%%%%%
    \symnum{3}{\le} & \  \sigma_k (1-\rho_B) \mathbb{E} \Bigg[ \frac{4 s_k^2}{\rho_M \rho_B} \left\| \nabla f(x_k) - \widetilde{\nabla}_k \right\|^2 + \frac{\rho_M \rho_B}{16} \|z_k - z_{k-1}\|^2 + s_k \left\langle \nabla f(x_k) - \mathbb{E}_{k-1} \widetilde{\nabla}_k, z_{k-1} - x^* \right\rangle \Bigg].
\end{align}
Equality \numcirc{1} is due to the MSEB property. We are able to pass the conditional expectation into the second inner product in \numcirc{2} because $z_{k-1}$ is independent of $\widetilde{\nabla}_k$ conditioned on the first $k-2$ iterates, and inequality \numcirc{3} is Young's. We can repeat this process once more, applying the MSEB property to obtain
\begin{align}
    %%%%%%%%%%%%%%%%%%%%%%%%%%%%%%%%%%%%%%%%%%%%%%%%%%%%%%%%%%%%%%
    & \sigma_k (1-\rho_B) \mathbb{E} \Bigg[ \frac{4 s_k^2}{\rho_M \rho_B} \left\| \nabla f(x_k) - \widetilde{\nabla}_k \right\|^2 + \frac{\rho_M \rho_B}{16} \|z_k - z_{k-1}\|^2 \notag \\
    & + s_k (1-\rho_B) \left\langle \nabla f(x_{k-1}) - \widetilde{\nabla}_{k-1}, z_{k-1} - x^* \right\rangle \Bigg] \\
    %%%%%%%%%%%%%%%%%%%%%%%%%%%%%%%%%%%%%%%%%%%%%%%%%%%%%%%%%%%%%%
    \le & \sigma_k (1-\rho_B) \mathbb{E} \Bigg[ \frac{4 s_k^2}{\rho_M \rho_B} \left\| \nabla f(x_k) - \widetilde{\nabla}_k \right\|^2 + \frac{4 s_k^2 (1 - \rho_B)}{\rho_M \rho_B} \left\| \nabla f(x_{k-1}) - \widetilde{\nabla}_{k-1} \right\|^2 \notag \\
    & + \frac{\rho_M \rho_B}{16} \left(\|z_k - z_{k-1}\|^2 + (1 - \rho_B) \|z_{k-1} - z_{k-2}\|^2 \right) \notag \\
    & + s_k (1-\rho_B) \left\langle \nabla f(x_{k-1}) - \widetilde{\nabla}_{k-1}, z_{k-2} - x^* \right\rangle \Bigg] \\
    %%%%%%%%%%%%%%%%%%%%%%%%%%%%%%%%%%%%%%%%%%%%%%%%%%%%%%%%%%%%%%
    \symnum{4}{\le} & (1-\rho_B) \mathbb{E} \Bigg[ \frac{4 \sigma_k s_k^2}{\rho_M \rho_B} \left\| \nabla f(x_k) - \widetilde{\nabla}_k \right\|^2 \notag \\
    & + \frac{4 \sigma_{k-1} s_{k-1}^2 \left( 1 - \frac{\rho_B}{2} \right)}{\rho_M \rho_B} \left\| \nabla f(x_{k-1}) - \widetilde{\nabla}_{k-1} \right\|^2 + \frac{\rho_M \rho_B}{16} \Big(\sigma_k \|z_k - z_{k-1}\|^2 \notag \\
    & + \sigma_{k-1} \left(1 - \frac{\rho_B}{2} \right) \|z_{k-1} - z_{k-2}\|^2 \Big) \notag \\
    & + \sigma_k s_k (1-\rho_B) \left\langle \nabla f(x_{k-1}) - \widetilde{\nabla}_{k-1}, z_{k-2} - x^* \right\rangle \Bigg].
\end{align}
Inequality \numcirc{4} uses our hypotheses on the decrease of $\sigma_k s_k^2$ and $\sigma_k$. This is a recursive inequality, and expanding the recursion yields
\begin{align}
    & \sigma_k s_k \mathbb{E} \left\langle \nabla f(x_{k+1}) - \E \widetilde{\nabla}_{k+1}, z_k - x^* \right\rangle \\
    \le & (1-\rho_B) \sum_{\ell=1}^k \sigma_\ell \mathbb{E} \Bigg[ \frac{4 s^2_\ell (1-\frac{\rho_B}{2})^{k-\ell}}{\rho_M \rho_B} \left\| \nabla f(x_\ell) - \widetilde{\nabla}_\ell \right\|^2 \notag \\
    & \quad \quad\quad\quad\quad\quad\quad\quad\quad\quad\quad\quad\quad\quad+ \frac{\rho_M \rho_B (1-\frac{\rho_B}{2})^{k-\ell}}{16} \|z_\ell - z_{\ell-1}\|^2 \Bigg].
\end{align}
The above uses the fact that $\widetilde{\nabla}_1 = \nabla f(x_1)$, so the inner product $\langle \nabla f(x_{1}) - \widetilde{\nabla}_{1},$ $ z_0 - x^* \rangle = 0$. Taking the sum over the iterations $k=0$ to $k=T-1$, we apply Lemma \ref{lem:tech} to simplify this bound.
\begin{align}
\label{thm1:bias2}
    & \sum_{k=0}^{T-1} \sigma_k s_k \mathbb{E} \left\langle \nabla f(x_{k+1}) - \E \widetilde{\nabla}_{k+1}, z_k - x^* \right\rangle \notag \\
    %%%%%%%%%%%%%%%%%%%%%%%%%%%%%%%%%%%%%%%%%%%%%%%%%
    \le & (1-\rho_B) \sum_{k=1}^{T-1} \sum_{\ell=1}^k \sigma_\ell \mathbb{E} \Big[ \frac{4 s^2_\ell (1-\frac{\rho_B}{2})^{k-\ell}}{\rho_M \rho_B} \left\| \nabla f(x_\ell) - \widetilde{\nabla}_\ell \right\|^2 \notag \\
    & \quad \quad\quad\quad\quad\quad\quad\quad\quad\quad\quad\quad\quad\quad+ \frac{\rho_M \rho_B (1-\frac{\rho_B}{2})^{k-\ell}}{16} \|z_\ell - z_{\ell-1}\|^2 \Big] \\
    %%%%%%%%%%%%%%%%%%%%%%%%%%%%%%%%%%%%%%%%%%%%%%
    \symnum{1}{\le} & (1-\rho_B) \sum_{k=0}^{T-1} \sigma_k \mathbb{E} \left[ \frac{8 s^2_k}{\rho_B^2 \rho_M} \left\| \nabla f(x_{k+1}) - \widetilde{\nabla}_{k+1} \right\|^2 + \frac{\rho_M}{8} \|z_{k+1} - z_k\|^2 \right] \\
    \symnum{2}{=} & (1-\rho_B) \sum_{k=0}^{T-1} \sigma_k \mathbb{E} \left[ \frac{8 s^2_k}{\rho_B^2 \rho_M} \left\| \nabla f(x_{k+1}) - \widetilde{\nabla}_{k+1} \right\|^2 + \frac{\rho_M}{8 \tau_k^2} \|x_{k+1} - y_{k+1}\|^2 \right].
\end{align}
Inequality \numcirc{1} follows from Lemma \ref{lem:tech}, and equality \numcirc{2} is the identity $y_{k+1} - x_{k+1} = \tau_k (z_{k+1} - z_k)$.
\end{proof}

This bound on the bias term includes the MSE, so to complete our bound on the bias term, we must combine Lemma \ref{lem:biasbound} with the following lemma.

\begin{lemma}[\textbf{MSE Bound}]
\label{lem:msebound}
    Suppose the stochastic gradient estimator $\widetilde{\nabla}_{k+1}$ satisfies the \emph{MSEB}$(M_1,M_2,\rho_M,$ $\rho_B,\rho_F)$ property, let $\rho = \min\{ \rho_M, \rho_B, \rho_F \}$, and let $\{s_k\}$ be any non-negative sequence satisfying $s^2_k \left( 1 - \rho \right) \le s^2_{k-1} \left( 1 - \frac{\rho}{2} \right)$. For convenience, define $\Theta_2 = \frac{M_1 \rho_F + 2 M_2}{\rho_M \rho_F}$. The MSE of the gradient estimator is bounded as
    \begin{align}
        & \sum_{k=0}^{T-1} s_k^2 \mathbb{E} \|\nabla f(x_{k+1}) - \widetilde{\nabla}_{k+1}\|^2 \le \sum_{k=0}^{T-1} 4 \Theta_2 L s_k^2 \mathbb{E} \left[ 2 D_f(y_k,x_{k+1}) + L \| x_{k+1} - y_{k+1} \|^2 \right]
    \end{align}
\end{lemma}

\begin{proof}
    First, we derive a bound on the sequence $\mathcal{F}_k$ arising in the MSEB property. Taking the sum from $k=0$ to $k=T-1$,
\begin{align}
    \sum_{k=0}^{T-1} s_k^2 \mathcal{F}_k \le & \sum_{k=0}^{T-1} \sum_{\ell=0}^k \frac{M_2 s_k^2 (1-\rho_F)^{k - \ell}}{n} \sum_{i=1}^n \mathbb{E} \| \nabla f_i(x_{\ell+1}) - \nabla f_i(x_\ell) \|^2 \\
    %%%%%%%%%%%%%%%%%%%%%%%%%%%%%%%%%%%%%%
    \symnum{1}{\le} & \sum_{k=0}^{T-1} \sum_{\ell=0}^k \frac{M_2 s_\ell^2 (1-\frac{\rho_F}{2})^{k - \ell}}{n} \sum_{i=1}^n \mathbb{E} \| \nabla f_i(x_{\ell+1}) - \nabla f_i(x_\ell) \|^2 \\
    %%%%%%%%%%%%%%%%%%%%%%%%%%%%%%%%%%%%%%%%
    \symnum{2}{\le} & \sum_{k=0}^{T-1} \frac{2 M_2 s_k^2}{n \rho_F} \sum_{i=1}^n \mathbb{E} \| \nabla f_i(x_{k+1}) - \nabla f_i(x_k) \|^2.
\end{align}
Inequality \numcirc{1} uses the fact that $s_k^2 (1-\rho_F) \le s_{k-1}^2\left( 1-\frac{\rho_F}{2} \right)$, and \numcirc{2} uses Lemma \ref{lem:tech}. With this bound on $\mathcal{F}_k$, we proceed to bound $\mathcal{M}_k$ in a similar fashion.
\begin{align}
    & \sum_{k=0}^{T-1} s_k^2 \mathbb{E} \|\nabla f(x_{k+1}) - \widetilde{\nabla}_{k+1}\|^2 \\
    %%%%%%%%%%%%%%%%%%%%%%%%%%%%%%%%%%%%%%%%%
    \le & \sum_{k=0}^{T-1} \frac{M_1 s_k^2}{n}
    \sum_{i=1}^n \mathbb{E} \|\nabla f_i(x_{k+1}) - \nabla f_i(x_k) \|^2 + s_k^2 \mathcal{F}_k + s_k^2 (1-\rho_M) \mathcal{M}_{k-1} \\
    %%%%%%%%%%%%%%%%%%%%%%%%%%%%%%%%%%%%%%%%%%%
    \le & \sum_{k=0}^{T-1} \frac{(M_1 \rho_F + 2 M_2) s_k^2}{n \rho_F} \sum_{i=1}^n \mathbb{E} \| \nabla f_i(x_{k+1}) - \nabla f_i(x_k) \|^2 + s_k^2 (1-\rho_M) \mathcal{M}_{k-1} \\
    \le & \sum_{k=0}^{T-1} \sum_{\ell=1}^k \frac{\Theta_2 s_k^2 (1-\rho_M)^{k - \ell} \rho_M}{n} \sum_{i=1}^n \mathbb{E} \| \nabla f_i(x_{\ell+1}) - \nabla f_i(x_\ell) \|^2 \\
    %%%%%%%%%%%%%%%%%%%%%%%%%%%%%%%%%%%%%%%%%%%%%%%%
    \symnum{1}{\le} & \sum_{k=0}^{T-1} \sum_{\ell=1}^k \frac{\Theta_2 s_\ell^2 (1-\frac{\rho_M}{2})^{k - \ell} \rho_M}{n} \sum_{i=1}^n \mathbb{E} \| \nabla f_i(x_{\ell+1}) - \nabla f_i(x_\ell) \|^2 \\
    %%%%%%%%%%%%%%%%%%%%%%%%%%%%%%%%%%%%%%%%%%%%%
    \symnum{2}{\le} & \sum_{k=0}^{T-1} \frac{2 \Theta_2 s_k^2}{n} \sum_{i=1}^n \mathbb{E} \| \nabla f_i(x_{k+1}) - \nabla f_i(x_k) \|^2 \\
    %%%%%%%%%%%%%%%%%%%%%%%%%%%%%%%%%%%%%%%%%%%%%%
    \symnum{3}{\le} & \sum_{k=0}^{T-1} \frac{4 \Theta_2 s_k^2}{n} \sum_{i=1}^n \mathbb{E} \| \nabla f_i(x_{k+1}) - \nabla f_i(y_k) \|^2 \notag \\
    & \quad \quad\quad\quad\quad\quad\quad\quad\quad\quad\quad\quad + \frac{4 \Theta_2 s_k^2}{n} \sum_{i=1}^n \mathbb{E} \| \nabla f_i(y_k) - \nabla f_i(x_k) \|^2 \\
    %%%%%%%%%%%%%%%%%%%%%%%%%%%%%%%%%%%%%%%%%%%%%%%
    \symnum{4}{\le} & \sum_{k=0}^{T-1} \left( 8 \Theta_2 L s_k^2 \mathbb{E} D_f(y_k,x_{k+1}) + 4 \Theta_2 L^2 s_k^2 \mathbb{E} \| x_k - y_k \|^2 \right).
\end{align}
Inequality \numcirc{1} uses $s_k^2(1-\rho_M) \le s_{k-1}^2\left( 1-\frac{\rho_M}{2} \right)$, \numcirc{2} uses Lemma \ref{lem:tech}, \numcirc{3} uses the inequality $\|a-c\|^2 \le 2 \|a - b\|^2 + 2 \|b - c\|^2$, and \numcirc{4} uses Lemma \ref{lem:2L} and the Lipschitz continuity of $\nabla f_i$.
\end{proof}

Lemmas \ref{lem:biasbound} and \ref{lem:msebound} show that it is possible to cancel the bias term and the MSE using the non-negative terms appearing in the inequality of Lemma \ref{lem:main}. Without these terms, we can telescope this inequality over several iterations and prove accelerated convergence rates. We are now prepared to prove Theorems \ref{thm:main1} and \ref{thm:main2}.

\renewcommand{\proofname}{\textbf{Proof of Theorem \ref{thm:main1}}}

\begin{proof}

We set $\mu = 0$ in the inequality of Lemma \ref{lem:main}, apply the full expectation operator, and sum the result over the iterations $k=0$ to $k=T-1$.
\begin{align}
\label{thm1:main}
0 \le & \frac{1}{2} \|z_0 - x^*\|^2 - \frac{1}{2} \mathbb{E} \|z_T - x^*\|^2 + \sum_{k=0}^{T-1} \mathbb{E} \Big[ \frac{\gamma_k (1-\tau_k)}{\tau_k} F(y_k) - \frac{\gamma_k}{\tau_k} F(y_{k+1}) \notag \\
& + \gamma_k F(x^*) + \frac{\gamma_k}{\tau_k} \left( \frac{L}{2} - \frac{1}{4 \tau_k \gamma_k} \right) \|x_{k+1} - y_{k+1}\|^2 - \frac{\gamma_k (1-\tau_k)}{\tau_k} D(y_k,x_{k+1}) \notag \\
& + \gamma_k \left\langle \nabla f(x_{k+1}) - \widetilde{\nabla}_{k+1}, z_k - x^* \right\rangle + \gamma_k^2 \|\nabla f(x_{k+1}) - \widetilde{\nabla}_{k+1}\|^2 \Big].
\end{align}
We bound the terms in the final line, beginning with the bias term. Our choice for $\gamma_k$ satisfies $\gamma_k^2 \left( 1 - \rho \right) \le \gamma_{k-1}^2 \left( 1 - \frac{\rho}{2} \right)$, so with $s_k = \gamma_k$ and $\sigma_k = 1$, we apply Lemma \ref{lem:biasbound}. This gives

\begin{align}
    0 \le & \frac{1}{2} \|z_0 - x^*\|^2 + \sum_{k=0}^{T-1} \mathbb{E} \Bigg[ \frac{\gamma_k (1-\tau_k)}{\tau_k} F(y_k) - \frac{\gamma_k}{\tau_k} F(y_{k+1}) + \gamma_k F(x^*) \notag \\
    & + \left( \frac{\gamma_k}{\tau_k} \left( \frac{L}{2} - \frac{1}{4 \tau_k \gamma_k} \right) + \frac{(1-\rho_B) \rho_M}{8 \tau_k^2} \right) \|x_{k+1} - y_{k+1}\|^2 \notag \\
    & - \frac{\gamma_k (1-\tau_k)}{\tau_k} D(y_k,x_{k+1}) + \gamma_k^2 \Theta_1 \|\nabla f(x_{k+1}) - \widetilde{\nabla}_{k+1}\|^2 \Bigg],
\end{align}
where we have dropped the term $-1/2 \mathbb{E} \|z_T - x^*\|^2$ because it is non-positive. Applying Lemma \ref{lem:msebound} to bound the MSE, we have
\begin{align}
\label{eq:thm1main}
    0 \le & \frac{1}{2} \|z_0 - x^*\|^2 + \sum_{k=0}^{T-1} \mathbb{E} \Big[ \frac{\gamma_k (1-\tau_k)}{\tau_k} F(y_k) - \frac{\gamma_k}{\tau_k} F(y_{k+1}) \notag \\
    & + \gamma_k F(x^*) + \left( 8 \gamma_k^2 L \Theta_1 \Theta_2 - \frac{\gamma_k (1-\tau_k)}{\tau_k} \right) D(y_k,x_{k+1}) \notag \\
    %%%%%%%%%%%%%%%%%%%%%%%%%%%%%%%%%%%%%%%%%%%%%%%%%%%%%%
    & + \left( \frac{\rho_M (1-\rho_B)}{8 \tau_k^2} + 4 \gamma_k^2 L^2 \Theta_1 \Theta_2 + \frac{\gamma_k}{\tau_k} \left( \frac{L}{2} - \frac{1}{4 \tau_k \gamma_k} \right) \right) \|x_{k+1} - y_{k+1}\|^2 \Big]. \notag \\
\end{align}

With the parameters set as in the theorem statement, it is clear that the final two lines of \eqref{eq:thm1main} are non-positive (see Appendix \ref{app:nonpos} for a proof).
This allows us to drop these lines from the inequality, leaving
\begin{align}
\label{eq:pretele}
0 \le & \frac{1}{2} \|z_0 - x^*\|^2 + \sum_{k=0}^{T-1} \mathbb{E} \left[ \frac{\gamma_k (1-\tau_k)}{\tau_k} F(y_k)
- \frac{\gamma_k}{\tau_k} F(y_{k+1}) + \gamma_k F(x^*) \right].
%%%%%%%%%%%%%%%%%%%%%%%%%%%%%%%%%%%%%%%%%%
\end{align}
Rewriting $\tau_k$ in terms of $\gamma_k$ shows that this is equivalent to
\begin{align}
0 \le & \frac{1}{2} \|z_0 - x^*\|^2 + \sum_{k=0}^{T-1} \mathbb{E} \left[ (c L \gamma_k^2 - \gamma_k) F(y_k)
- c L \gamma_k^2 F(y_{k+1}) + \gamma_k F(x^*) \right].
\end{align}
Our choice for $\gamma_k$ satisfies $c L \gamma_k^2 - \gamma_k = c L \gamma_{k-1}^2 - \frac{1}{4 c L}$, allowing the $F(y_k)$ terms to telescope. Hence, our inequality is equivalent to 
\begin{align}
0 &\le - c L \gamma_{T-1}^2 \mathbb{E} [F(y_T) - F(x^*)] - \frac{1}{4 c L} \sum_{k=1}^{T-1} \mathbb{E} \left[ F(y_k) - F(x^*) \right] + (c L \gamma_0^2 - \gamma_0) (F(y_0) - F(x^*)) + \frac{1}{2} \|z_0 - x^*\|^2.
\end{align}
Using the facts that $c L \gamma_{T-1}^2 = \frac{(T + \nu + 3)^2}{4 c L}$, $c L \gamma_0^2 - \gamma_0 = \frac{(\nu+2)(\nu+4)}{4 c L}$, and $F(y_k) \le F(x^*)$, we have
\begin{align}
& \frac{(T + \nu + 3)^2}{4 c L} \mathbb{E} [ F(y_T) - F(x^*) ] \le \frac{(\nu+2)(\nu+4)}{4 c L} (F(y_0) - F(x^*)) + \frac{1}{2} \|z_0 - x^*\|^2.
%%%%%%%%%%%%%%%%%%%%%%%%%%%%%%%%%%%%%%%%%%%%%%%%%%%%%%%%%%%%%%%%
\end{align}
This proves the assertion.

\end{proof}

A similar argument proves an accelerated linear convergence rate when strong convexity is present.

\renewcommand{\proofname}{\textbf{Proof of Theorem \ref{thm:main2}}}

\begin{proof}

We recall the inequality of Lemma \ref{lem:main}.
\begin{align}
& \frac{\gamma}{\tau} \left( F(y_{k+1}) - F(x^*) \right) + \frac{(1+\mu \gamma)}{2} \|z_{k+1} - x^*\|^2 \notag \\
\le & \frac{\gamma(1-\tau)}{\tau} \left(F(y_k) - F(x^*)\right) + \frac{1}{2} \|z_k - x^*\|^2 + \gamma^2 \|\widetilde{\nabla}_{k+1} - \nabla f(x_{k+1})\|^2 \notag \\
& + \frac{\gamma}{\tau} \left( \frac{L}{2} - \frac{1}{4 \tau \gamma} \right) \|x_{k+1} - y_{k+1}\|^2 - \frac{\gamma(1-\tau)}{\tau} D_f(y_k,x_{k+1}) \notag \\
& + \gamma \left\langle \nabla f(x_{k+1}) - \widetilde{\nabla}_{k+1}, z_k - x^* \right\rangle.
\end{align}
By our choice of $\gamma$ and $\tau$, we have
\begin{equation}
    \frac{\gamma}{\tau} \left( \frac{\gamma (1 - \tau)}{\tau} \right)^{-1} = \frac{1}{1-\tau} \ge 1 + \tau = 1 + \mu \gamma.
\end{equation}
Therefore, we can extract a factor of $(1+\mu \gamma)$ from the left.
\begin{align}
& (1+\mu \gamma) \Bigg( \frac{\gamma(1-\tau)}{\tau} \left( F(y_{k+1}) - F(x^*) \right) + \frac{1}{2} \|z_{k+1} - x^*\|^2 \Bigg) \notag \\
\le & \frac{\gamma(1-\tau)}{\tau} \left(F(y_k) - F(x^*)\right) + \frac{1}{2} \|z_k - x^*\|^2 + \gamma^2 \|\widetilde{\nabla}_{k+1} - \nabla f(x_{k+1})\|^2 \notag \\
& + \frac{\gamma}{\tau} \left( \frac{L}{2} - \frac{1}{4 \tau \gamma} \right) \|x_{k+1} - y_{k+1}\|^2 - \frac{\gamma(1-\tau)}{\tau} D_f(y_k,x_{k+1}) \notag \\
& + \gamma \left\langle \nabla f(x_{k+1}) - \widetilde{\nabla}_{k+1}, z_k - x^* \right\rangle.
\end{align}
Multiplying this inequality by $(1+\mu \gamma)^k$, summing over iterations $k=0$ to $k = T-1$, and applying the full expectation operator, we obtain the bound
\begin{align}
\label{eq:thm2}
& (1+\mu \gamma)^T \mathbb{E} \Bigg[ \frac{\gamma(1-\tau)}{\tau} \left( F(y_T) - F(x^*) \right) + \frac{1}{2} \|z_T - x^*\|^2 \Bigg] \notag \\
\le & \frac{\gamma(1-\tau)}{\tau} \left(F(y_0) - F(x^*)\right) + \frac{1}{2} \|z_0 - x^*\|^2 + \sum_{k=0}^{T-1} (1+\mu \gamma)^k \mathbb{E} \Big[ \gamma^2 \|\widetilde{\nabla}_{k+1} - \nabla f(x_{k+1})\|^2 \\
& + \frac{\gamma}{\tau} \left( \frac{L}{2} - \frac{1}{4 \tau \gamma} \right) \|x_{k+1} - y_{k+1}\|^2 - \frac{\gamma(1-\tau)}{\tau} D_f(y_k,x_{k+1}) + \gamma \left\langle \nabla f(x_{k+1}) - \widetilde{\nabla}_{k+1}, z_k - x^* \right\rangle \Big].
\end{align}
As in the proof of Theorem \ref{thm:main1}, we bound the bias term and the MSE using Lemmas \ref{lem:biasbound} and \ref{lem:msebound}, respectively. To apply Lemma \ref{lem:biasbound}, we let $\sigma_k = (1+\mu \gamma)^k$ and $s_k = \gamma$. These choices are appropriate because $(1 + \mu \gamma)^k (1-\rho) \le (1 + \mu \gamma)^{k-1} (1 - \frac{\rho}{2})$ due to the fact that $\mu \gamma \le \rho/2$.

Combining these bounds with \eqref{eq:thm2}, we have
\begin{align}
& (1+\mu \gamma)^T \mathbb{E} \Bigg[ \frac{\gamma(1-\tau)}{\tau} \left( F(y_T) - F(x^*) \right) + \frac{1}{2} \|z_T - x^*\|^2 \Bigg] \notag \\
\le & \frac{\gamma(1-\tau)}{\tau} \left(F(y_0) - F(x^*)\right) + \frac{1}{2} \|z_0 - x^*\|^2 \\
& + \sum_{k=0}^{T-1} (1+\mu \gamma)^k \mathbb{E} \Bigg[ \left( 8 \gamma^2 L \Theta_1 \Theta_2 - \frac{\gamma (1-\tau)}{\tau} \right) D(y_k,x_{k+1}) \\
& + \left( \frac{\rho_M (1-\rho_B)}{8 \tau^2} + 4 \gamma^2 L^2 \Theta_1 \Theta_2 + \frac{\gamma}{\tau} \left( \frac{L}{2} - \frac{1}{4 \tau \gamma} \right) \right) \|x_{k+1} - y_{k+1}\|^2 \Big].
\end{align}
The parameter settings in the theorem statement ensure the final two lines are non-positive (see Appendix \ref{app:nonpos} for details). This gives
\begin{align}
\frac{1}{2} \mathbb{E} \|z_T - x^*\|^2 & \le (1 + \mu \gamma)^{-T} \left( \frac{\gamma(1-\tau)}{\tau} \left( F(y_0) - F(x^*) \right) + \frac{1}{2} \|z_0 - x^*\|^2 \right) \\
& \le \left(1 + \min \left\{ \sqrt{\frac{\mu}{L c}}, \frac{\rho}{2} \right\} \right)^{-T} \left( \frac{1}{\mu} \left(F(y_0) - F(x^*)\right) + \frac{1}{2} \|z_0 - x^*\|^2 \right),
\end{align}
which is the desired result.
\end{proof}

\renewcommand{\proofname}{Proof}

\section{Convergence Rates for Specific Estimators}
\label{sec:convrates}

In light of Theorems \ref{thm:main1} and \ref{thm:main2}, we must only establish suitable bounds on the MSE and bias terms of a gradient estimator to prove accelerated convergence rates for Algorithm \ref{alg:1}. We consider four variance-reduced gradient estimators: SAGA, SVRG, SARAH, and SARGE, beginning with the unbiased estimators. We defer proofs to the Appendix. To preserve the generality of our framework, we have not optimised the constants appearing in the presented convergence rates.

\begin{theorem}[SAGA Convergence Rates]
\label{thm:saga}
    When using the SAGA gradient estimator in Algorithm \ref{alg:1}, set $b \le 4 \sqrt{2} n^{2/3}$, $\gamma_k = \frac{b^3 (k+\frac{4 n}{b}+4)}{192 n^2 L}$, and $\tau_k = \frac{b^3}{96 n^2 L \gamma_k}$. After $T$ iterations, the suboptimalty at $y_T$ satisfies
    \begin{align}
        \mathbb{E} F(y_T) - F(x^*) \le \frac{(\frac{4 n}{b}+2)(\frac{4 n}{b}+4) K_1}{(T + \frac{4 n}{b} + 3)^2},
    \end{align}
    where
    \begin{equation}
        K_1 = \left( F(y_0) - F(x^*) + \frac{192 n^2 L}{b^3(\frac{4 n}{b}+2)(\frac{4 n}{b}+4)} \|z_0 - x^*\|^2 \right).
    \end{equation}
    If $g$ is $\mu$-strongly convex, set $\gamma = \min \left\{ \frac{b^{3/2}}{4 n \sqrt{6 \mu L}}, \frac{b}{4 n \mu} \right\}$ and $\tau = \mu \gamma$. After $T$ iterations, the point $z_T$ satisfies
    \begin{align}
        \mathbb{E} \|z_T - x^*\|^2 \le \left(1+ \min \left\{ \frac{b^{3/2} \sqrt{\mu}}{4 n \sqrt{6 L}}, \frac{b}{4 n} \right\} \right)^{-T} K_2,
    \end{align}
    where $K_2$ is defined as in Theorem \ref{thm:main2}.
\end{theorem}

It is enlightening to compare these rates to existing convergence rates for full and stochastic gradient methods. In the non-strongly convex setting, our convergence rate is $\mathcal{O}\left(n^2/T^2\right)$, matching that of Katyusha. As with Katyusha, this rate could be improved for SVRG using the epoch-doubling procedure in SVRG++ (see Allen-Zhu, 2018 \cite{svrgplusplus} for further details). In the strongly convex case, if $F$ is poorly conditioned so that $L/\mu \ge \mathcal{O}( b )$, we prove linear convergence at the rate $\mathcal{O} \left( \left(1 + \frac{b^{3/2} \sqrt{\mu}}{n \sqrt{L}} \right)^{-T} \right)$. With $b = n^{2/3}$, this rate matches the convergence rate of inertial forward-backward on the same problem (i.e., the rate is independent of $n$), but we require only $n^{2/3}$ stochastic gradient evaluations per iteration compared to the $n$ evaluations that full gradient methods require. This is reminiscent of the results of \cite{reddi,zhunoncon}, where the authors show that SAGA and SVRG achieve the same convergence rate as full gradient methods on non-convex problems using only $n^{2/3}$ stochastic gradient evaluations per iteration. This is slightly worse than the results proven for Katyusha, which requires $\mathcal{O}(\sqrt{n})$ stochastic gradient evaluation per iteration to match the convergence rate of full-gradient methods.

The analogous convergence guarantees for SVRG are included in Theorem~\ref{thm:svrg}.

\begin{theorem}[SVRG Convergence Rates]
\label{thm:svrg}
    When using the SVRG gradient estimator in Algorithm \ref{alg:1}, set $b \le 32 p^2$, $\gamma_k = \frac{b (k+4 p+4)}{192 p^2 L}$, and $\tau_k = \frac{b}{ 96 p^2 L \gamma_k}$. After $T$ iterations, the suboptimalty at $y_T$ satisfies
    \begin{align}
        \mathbb{E} F(y_T) - F(x^*) \le \frac{(4 p+2)(4 p+4) K_1}{(T + 4 p + 3)^2},
    \end{align}
    where
    \begin{equation}
        K_1 = \left( F(y_0) - F(x^*) + \frac{192 p^2 L}{b (4 p+2)(4 p+4)} \|z_0 - x^*\|^2 \right).
    \end{equation}
    If $g$ is $\mu$-strongly convex, set $\gamma = \min \left\{ \frac{\sqrt{b}}{4 p \sqrt{6 \mu L}}, \frac{1}{4 p \mu} \right\}$ and $\tau = \mu \gamma$. After $T$ iterations, the point $z_T$ satisfies
    \begin{align}
        \mathbb{E} \|z_T - x^*\|^2 \le \left(1 + \min \left\{ \frac{\sqrt{b \mu}}{4 p \sqrt{6 L}}, \frac{b}{4 p} \right\} \right)^{-T} K_2,
    \end{align}
    where $K_2$ is defined as in Theorem \ref{thm:main2}.
\end{theorem}

The convergence rates for SVRG are similar to the rates for SAGA if $p$ and $b$ are chosen appropriately. In the strongly convex case, setting $b = p^2$ allows SVRG to match the convergence rate of full gradient methods, and the expected number of stochastic gradient evaluations per iteration is $n/p + b$. To minimise the number of stochastic gradient evaluations while maintaining the convergence rate of full gradient methods, we set $p = \mathcal{O}(n^{1/3})$, showing that Algorithm \ref{alg:1} using the SVRG gradient estimator achieves the same convergence rate as full gradient methods using only $\mathcal{O}(n^{2/3})$ stochastic gradient evaluations per iteration.

\begin{remark}
    The above discussion shows that when using the SAGA gradient estimator on a strongly convex objective with $b = \mathcal{O}(n^{2/3})$ and $\gamma = \mathcal{O}(1 / \sqrt{\mu L} )$, Algorithm \ref{alg:1} finds a point satisfying $\mathbb{E} \|z_T - x^*\|^2 \le \epsilon$ in $\mathcal{O}(n^{2/3} \sqrt{\kappa} \log(1 / \epsilon))$ iterations. This is compared to the complexity lower bound of $\mathcal{O}(\sqrt{n \kappa} \log(1 / \epsilon) )$ achieved by Katyusha \cite{srebrocomplexity}. SVRG has a similar complexity when $b = n^{2/3}$ and $p = n^{1/3}$.
\end{remark}

The SARAH gradient estimator is similar to the SVRG estimator, as both estimators require the full gradient to be computed periodically. SARAH differs from SVRG by using previous estimates of the gradient to inform future estimates. The recursive nature of the estimator seems to decrease its MSE, which can be observed in experiments and in theory \cite{sarah,techrepo}. However, this comes at the cost of introducing bias into the estimator.

Biased stochastic gradient methods are underdeveloped compared to their unbiased counterparts. The convergence proofs for biased algorithms are traditionally complex and difficult to generalize (see \cite{SAG}, for example), and proximal support has only recently been extended to the biased algorithms SARAH and SARGE, as well as biased versions of SAGA and SVRG in the convex setting \cite{techrepo}. It is difficult to determine conclusively if the negative effect of the bias outweighs the benefits of a lower MSE. We show that Algorithm \ref{alg:1} is able to achieve accelerated rates of convergence using biased estimators as well, beginning with the SARAH estimator.

\begin{theorem}[SARAH Convergence Rates]
\label{thm:sarah}
When using the SARAH gradient estimator in Algorithm \ref{alg:1}, set $\gamma_k = \frac{k + 2 p + 4}{288 p^4 L}$, and $\tau_k = \frac{1}{144 p^4 L \gamma_k}$. After $T$ iterations, the suboptimalty at $y_T$ satisfies
\begin{align}
    \mathbb{E} F(y_T) - F(x^*) \notag \le \frac{(2 p + 2)(2 p + 4) K_1}{(T + 2 p + 3)^2}.
\end{align}
where
\begin{equation}
    K_1 = \left( F(y_0) - F(x^*) + \frac{288 p^4 L}{(2 p + 2)(2 p + 4)} \|z_0 - x^*\|^2 \right).
\end{equation}
If $g$ is $\mu$-strongly convex, set $\gamma = \min \left\{ \sqrt{ \frac{1}{144 p^4 \mu L}}, \frac{1}{2 p \mu} \right\}$ and $\tau = \mu \gamma$. After $T$ iterations, the point $z_T$ satisfies
\begin{align}
    \mathbb{E} \|z_T - x^*\|^2 \le \left(1 + \min \left\{ \sqrt{\frac{\mu}{144 p^4 L}}, \frac{1}{2 p} \right\} \right)^{-T} K_2,
\end{align}
where $K_2$ is defined as in Theorem \ref{thm:main2}.
\end{theorem}

We provide a proof of this result in Appendix \ref{sec:sarah}. Theorem \ref{thm:sarah} shows that using the SARAH gradient estimator in Algorithm \ref{alg:1} achieves an optimal $\mathcal{O}\left( 1 / T^2 \right)$ convergence rate on convex objectives, but with $p = \mathcal{O}(n)$, the constant is a factor of $n^2$ worse than it is for accelerated SAGA, SVRG, and Katyusha. In the strongly convex case, setting $p = \mathcal{O} (n)$ and $b = \mathcal{O}(1)$ guarantees a linear convergence rate of $\mathcal{O} ( (1 + n^{-2} \sqrt{\mu/L} )^{-T} )$, achieving the optimal dependence on the condition number, but with a constant that is a factor of $n$ worse than accelerated SAGA and SVRG, and a factor of $n^{3/2}$ worse than Katyusha. Despite this dependence on $n$, experimental results, including those in Section \ref{sec:experiments} and \cite{sarah}, show that the SARAH gradient estimator exhibits competitive performance.

Finally, we provide convergence rates for the SARGE estimator. In \cite{techrepo}, the authors introduce the SARGE gradient estimator to mimic the recursive nature of SARAH but trade larger storage costs for a lower average per-iteration complexity, similar to the relationship between SAGA and SVRG. We prove in Appendix \ref{sec:sarge} that SARGE satisfies the MSEB property with similar constants to SARAH, and achieves similar convergence rates as well.

\begin{theorem}[SARGE Convergence Rates]
\label{thm:sarge}
    Let\footnote{Throughout this manuscript, we have sacrificed smaller constants for generality and ease of exposition, so the constant appearing in $c$ is not optimal.} $c = 86016 n^4/b^4$. 
    When using the SARGE gradient estimator in Algorithm \ref{alg:1}, set $\gamma_k = \frac{k+\frac{4 n}{b}+4}{2 c L}$ and $\tau_k = \frac{1}{c L \gamma_k}$. After $T$ iterations, the suboptimalty at $y_T$ satisfies
    \begin{align}
        \mathbb{E} F(y_T) - F(x^*) \le \frac{2 (\frac{2 n}{b} + 1)( \frac{2 n}{b} + 2) K_1}{(T + \frac{4 n}{b} + 3)^2},
    \end{align}
    where
    \begin{equation}
        K_1 = \left( F(y_0) - F(x^*) + \frac{86016 n^4}{b^4 (\frac{2 n}{b} + 1)( \frac{2 n}{b} + 2)} \|z_0 - x^*\|^2 \right).
    \end{equation}
    If $g$ is $\mu$-strongly convex, set $\gamma = \min \left\{ \frac{1}{\sqrt{c \mu L}}, \frac{b}{4 n \mu} \right\}$ and $\tau = \mu \gamma$. After $T$ iterations, the point $z_T$ satisfies
    \begin{align}
        \mathbb{E} \|z_T - x^*\|^2 \le \left(1+ \min \left\{ \frac{48 b^2 \sqrt{154 \mu}}{n^2 \sqrt{L}}, \frac{b}{4 n} \right\} \right)^{-T} K_2,
    \end{align}
    where $K_2$ is defined as in Theorem \ref{thm:main2}.
\end{theorem}

The convergence rates for SARGE are of the same order as the convergence rates for SARAH, even though SARGE requires fewer stochastic gradient evaluations per iteration on average.

Although our bound on the MSE of the SARAH and SARGE estimators is a factor of $n$ smaller than our bound on the MSE of the SAGA and SVRG estimators, the analytical difficulties due to the bias lead to a worse dependence on $n$. Nevertheless, SARAH and SARGE are competitive in practice, as we demonstrate in the following section.

\section{Numerical Experiments}
\label{sec:experiments}

\begin{figure}[!t]
\centering
\centering\captionsetup[subfloat]{labelfont=bf}
\subfloat[\texttt{australian}]{ \includegraphics[width=0.45\linewidth]{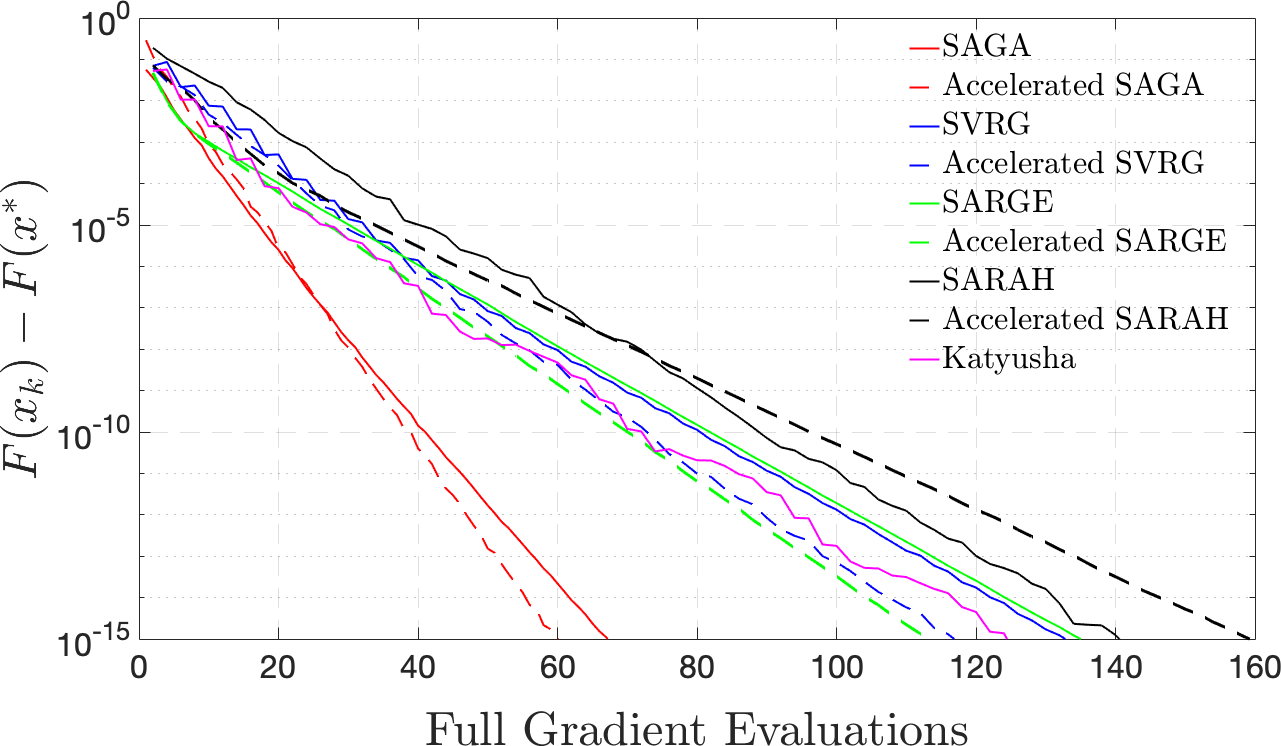} }  \hspace{8pt}
\subfloat[\texttt{mushrooms}]{ \includegraphics[width=0.45\linewidth]{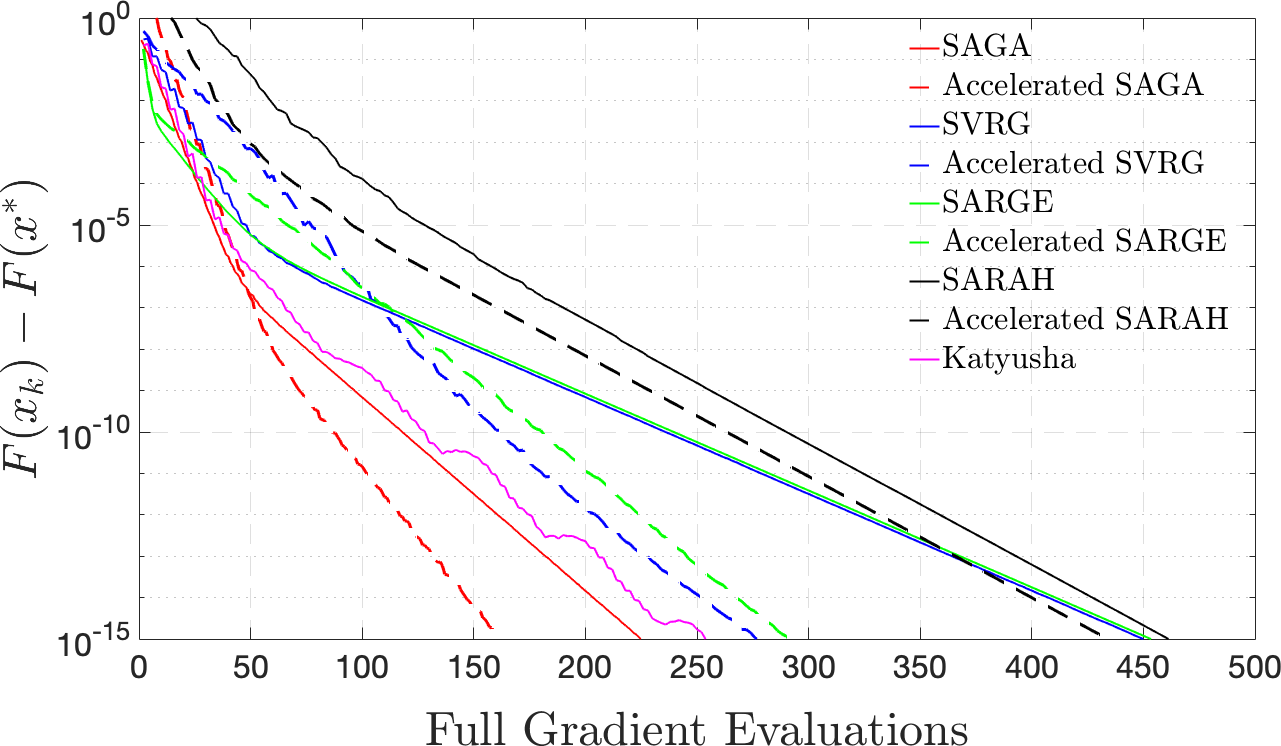} } \\ 
\subfloat[\texttt{phishing}]{ \includegraphics[width=0.45\linewidth]{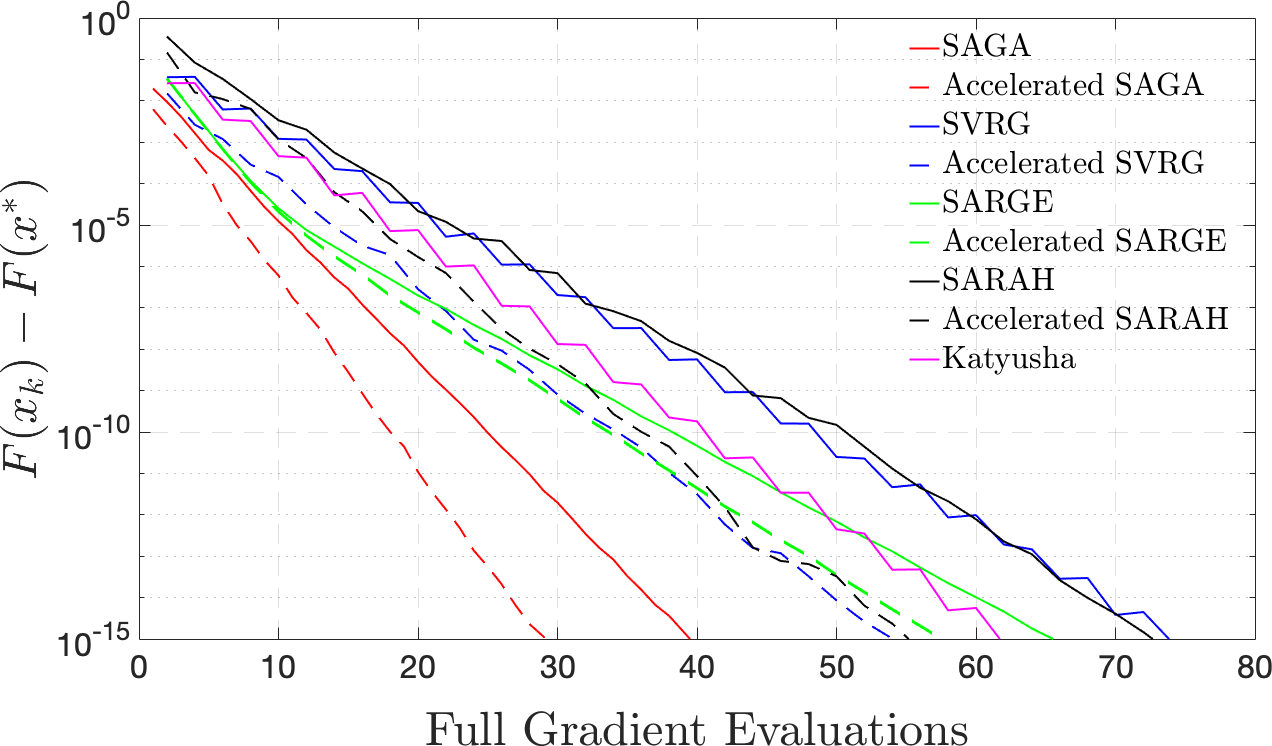} }   
\hspace{8pt}
\subfloat[\texttt{ijcnn1}]{ \includegraphics[width=0.45\linewidth]{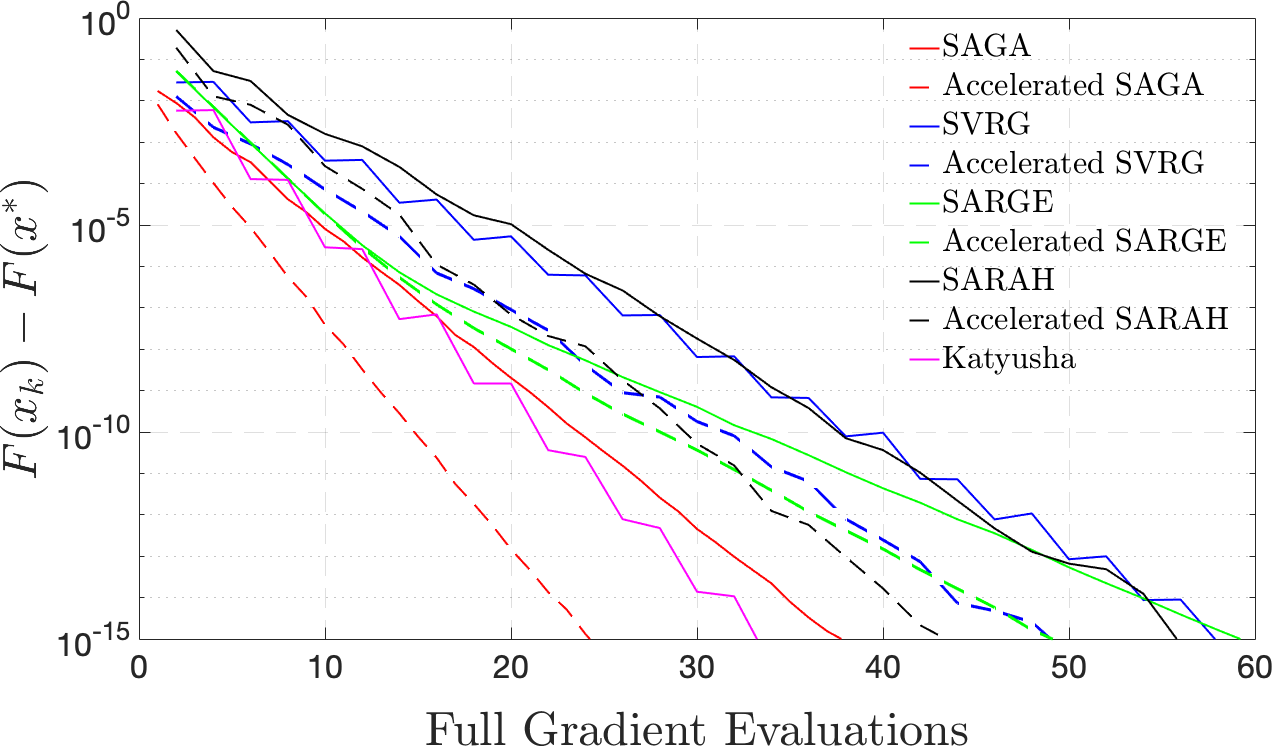} }  \\
%%%%%%%%
\caption{Performance comparison for solving ridge regression among different algorithms.}
\label{fig:ridge}
\end{figure}

To test our acceleration framework, we use it to accelerate SAGA, SVRG, SARAH, and SARGE on a series of ridge regression and LASSO tasks using the binary classification data sets \texttt{australian}, \texttt{mushrooms}, \texttt{phishing}, and \texttt{ijcnn1} from the LIBSVM\footnote{\url{https://www.csie.ntu.edu.tw/~cjlin/libsvmtools/datasets/}} database. We include Katyusha and Katyusha\textsuperscript{ns} for comparison as well. For SVRG and SARAH, we compare our accelerated variants that compute the full gradient probabilistically to the non-accelerated versions that compute the full gradient deterministically at the beginning of each epoch.

With feature vectors $a_i$ and labels $y_i$ for $i \in \{1,2,\cdots,n\}$, ridge regression and LASSO can be written as
\begin{equation}
    \min_{x \in \R^m} \quad \frac{1}{n} \sum_{i=1}^n (a_i^\top x - y_i)^2 + \lambda R(x),
\end{equation}
where $R \equiv \tfrac{1}{2} \|\cdot\|^2$ in ridge regression and $R \equiv \|\cdot\|_1$ for LASSO. Letting $g \equiv \lambda R$, it is clear that $g$ is $\lambda$-strongly convex in ridge regression and $g$ is not strongly convex for LASSO. In all our experiments, we rescale the value of the data to $[-1, 1]$. For ridge regression, we set $\lambda = 1 / n$, and for LASSO, we set $\lambda = 1 / \sqrt{n}$.

For accurate comparisons, we automate all our parameter tuning. For our experiments using ridge regression, we select the step size and momentum parameters from the set $\{ 1 / t : t \in \mathbb{N} \}$. For LASSO, we use the parameters suggested by Theorem \ref{thm:main1}, but we scale the step size by a constant $s \in \mathbb{N}$, and we rescale the momentum parameter so that $\tau_0 = 1 / 2$. We perform the same parameter-tuning procedure for Katyusha, and set the negative momentum parameter $\tau_2 = 1 / 2$ as suggested in \cite{katyusha} unless otherwise stated. In our accelerated variants of SVRG and SARAH, we set $p = \frac{1}{2 n}$, and for the non-accelerated variants and Katyusha, we set the epoch length to $2 n$. We use a batch size of $b = 1$ for all algorithms.

\begin{figure}[t]
\centering
\centering\captionsetup[subfloat]{labelfont=bf}
\subfloat[\texttt{australian}]{ \includegraphics[width=0.45\linewidth]{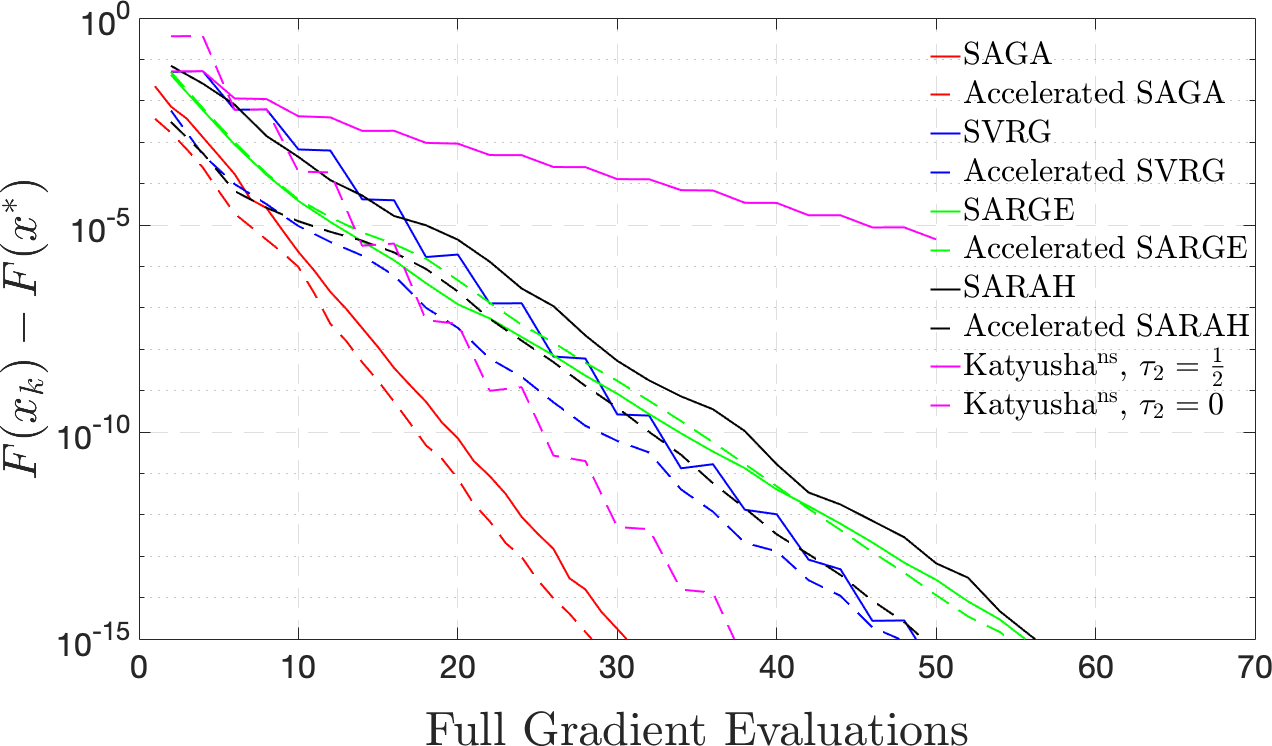} }  \hspace{8pt}
\subfloat[\texttt{mushrooms}]{ \includegraphics[width=0.45\linewidth]{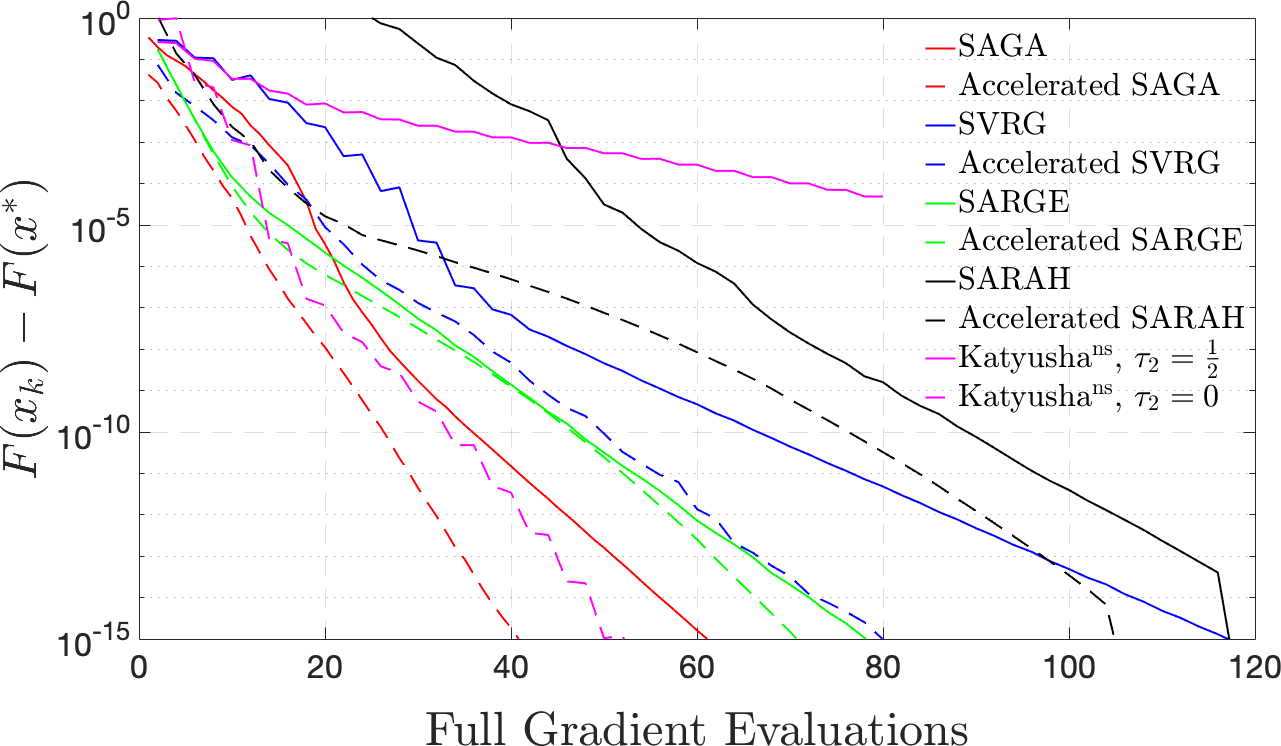} } \\ 
\subfloat[\texttt{phishing}]{ \includegraphics[width=0.45\linewidth]{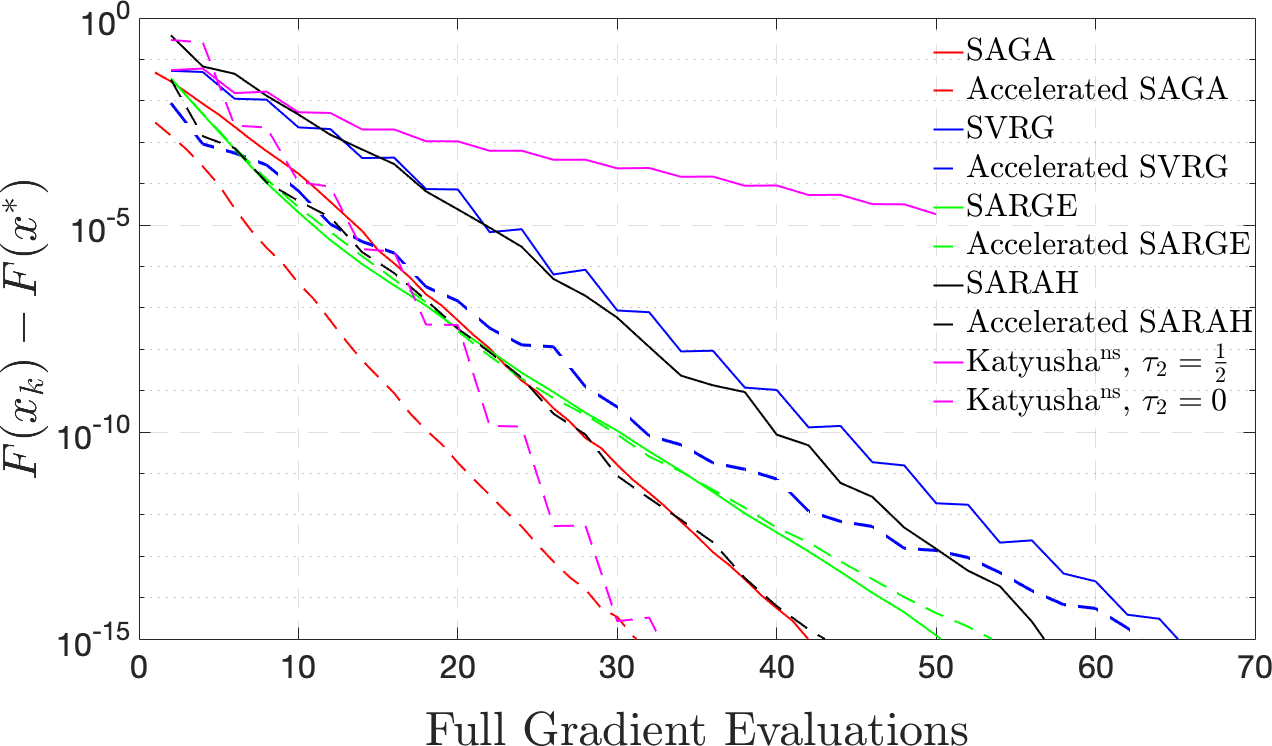} }   
\hspace{8pt}
\subfloat[\texttt{ijcnn1}]{ \includegraphics[width=0.45\linewidth]{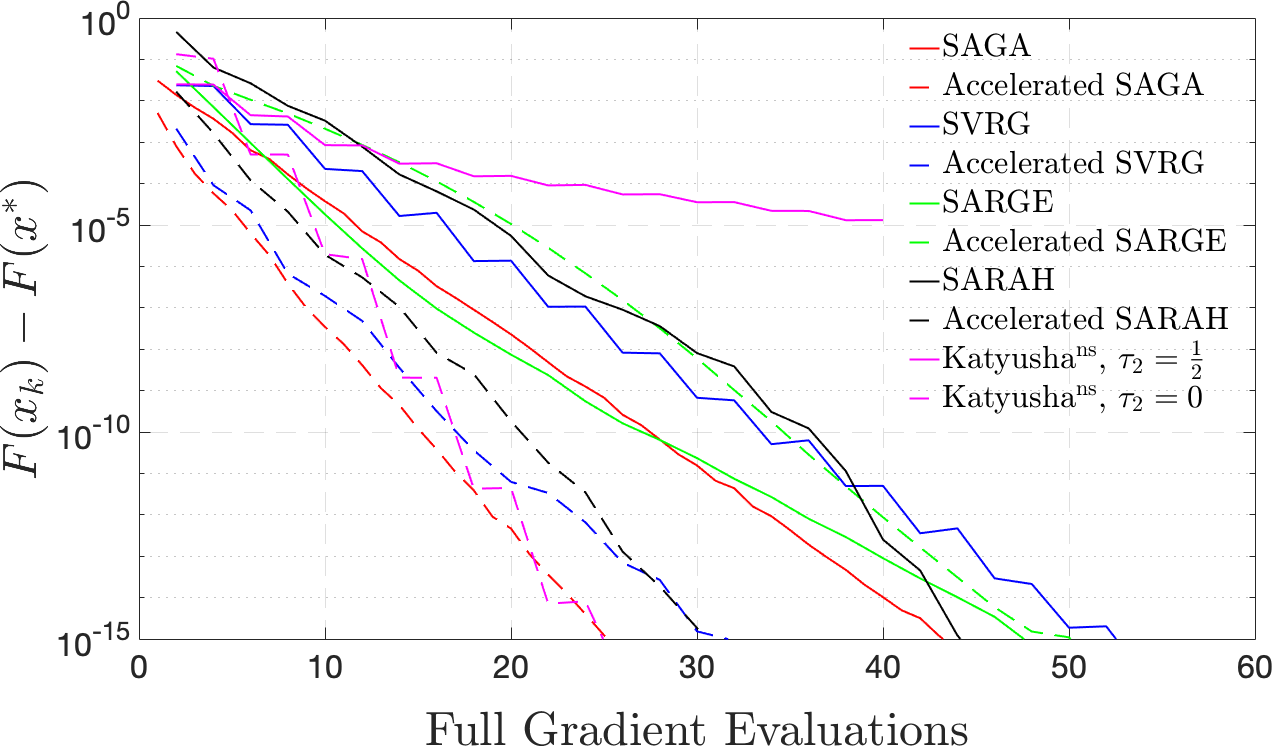} }  \\
%%%%%%%%
\caption{Performance comparison for solving LASSO among different algorithms. In Katyusha, the negative momentum parameters $\tau_2 = 0, \frac{1}{2}$ are not tuned.}
\label{fig:lasso}
\end{figure}

We measure performance with respect to the suboptimality $F(x_{k+1}) - F(x^*)$, where $x^*$ is a low-tolerance solution found using forward-backward. To fairly compare algorithms that require a different number of stochastic gradient evaluations per iteration, we report their performance with respect to the number of effective full gradient computations they perform on average each iteration. By this metric, SAGA performs $1 / n$ full gradient computations each iteration, while SVRG performs an average of $\frac{2}{n} + \frac{1}{2 n}$, for example.

Figures \ref{fig:ridge} and \ref{fig:lasso} display the median of 100 trials of ridge regression and LASSO, respectively. We observe the following trends:

\begin{itemize}
    \item Acceleration without negative momentum significantly improves the performance of SAGA, SVRG, SARAH, and SARGE in most cases. The improvement is least dramatic on the smallest data set, \texttt{australian}, and slightly less dramatic for the biased algorithms, SARAH and SARGE.
    \item Because they require only one stochastic gradient evaluation per iteration, SAGA and Accelerated SAGA require significantly less computation to achieve the same accuracy as other methods.
    \item In the strongly convex setting, Katyusha performs similarly to or better than SVRG with acceleration in most cases.
    \item In the non-strongly convex setting, Katyusha\textsuperscript{ns} performs much worse than other methods when using negative momentum. Without negative momentum, it performs much better than all algorithms except Accelerated SAGA. Because Katyusha without negative momentum is almost exactly the same algorithm as Accelerated SVRG, this improved performance is likely due to the second proximal step and additional step size $\eta$ in Katyusha. All of the algorithms presented in this work can adopt these features without changing their convergence rates.
\end{itemize}

\section{Conclusion}

Although acceleration is a widely used and an extensively researched technique in first-order optimisation, its application to stochastic gradient methods is still poorly understood. The introduction of negative momentum adds another layer of complexity to this line of research. Although algorithms using negative momentum enjoy fast convergence rates and strong performance when the parameters are tuned appropriately, it is unclear if negative momentum is necessary for acceleration. In this work, we propose a universal framework for accelerating stochastic gradient methods that does not rely on negative momentum.

Because our approach does not rely on negative momentum, it applies to a much broader class of stochastic gradient estimators. As long as the estimator admits natural bounds on its bias and MSE, it can be used in our framework to produce an accelerated stochastic gradient method with an optimal $1/T^2$ dependence on convex problems and an optimal $\sqrt{\kappa}$ dependence in the strongly convex setting. The bias and MSE of the estimator appear only in the constants of our convergence rates. From this perspective, negative momentum is effectively a variance-reduction technique, reducing the variance in the iterates to improve the dependence on $n$ in the convergence rates. A natural question for future research is whether there exist gradient estimators with smaller bias and MSE than SAGA, SVRG, SARAH, and SARGE that can be accelerated using our framework and admit a better dependence on $n$.

\section*{Acknowledgements}

C.-B.S. and M.J.E. acknowledge support from the EPSRC grant No. EP/S0260 45/1. C.-.B.S. acknowledges support from the Leverhulme Trust project ``Breaking the nonconvexity barrier'', the Philip Leverhulme Prize, the EPSRC grant No. EP/M00483X/1, the EPSRC Centre No. EP/N014588/1, the European Union Horizon 2020 research and innovation programmes under the Marie Skodowska-Curie grant agreement No. 777826 NoMADS and No. 691070 CHiPS, the Cantab Capital Institute for the Mathematics of Information and the Alan Turing Institute.

\bibliographystyle{acm}
\bibliography{main.bib}

\clearpage

\begin{appendix}

\section{One Technical Lemma}

\begin{lemma}
\label{lem:tech}
    Given a non-negative sequence $\sigma_k$, a constant $\rho \in [0,1]$, and an index $T \ge 1$, the following estimate holds:
    \begin{equation}
        \sum_{k = 1}^T \sum_{\ell = 1}^k (1 - \rho)^{k - \ell} \sigma_\ell \le \frac{1}{\rho} \sum_{k=1}^T \sigma_k.
    \end{equation}
\end{lemma}

\begin{proof}
    This follows from expanding the double-sum and computing a series:
    \begin{align}
        \sum_{k = 1}^T \sum_{\ell = 1}^k (1 - \rho)^{k - \ell} \sigma_\ell & = [ \sigma_1 ] + [ (1 - \rho) \sigma_1 + \sigma_2 ] + [ (1-\rho)^2 \sigma_1 + (1-\rho) \sigma_2 + \sigma_3] + \cdots \\
        %%%%%%%%%%%%%%%%%%%%%%%%%%%%%%%%%%
        & = [1 + (1-\rho) + (1-\rho)^2 + \cdots + (1-\rho)^{T-1}] \sigma_1 \\
        & \quad + [1 + (1-\rho) + (1-\rho)^2 + \cdots + (1-\rho)^{T-2}] \sigma_2 + \cdots \\
        %%%%%%%%%%%%%%%%%%%%%%%%%%%%%%%%%%%%
        & \le \left( \sum_{\ell = 0}^\infty (1 - \rho)^\ell \right) \left( \sum_{k=1}^T \sigma_k \right) \\
        %%%%%%%%%%%%%%%%%%%%%%%%%%%%%%%%%%%%%
        & = \frac{1}{\rho} \sum_{k=1}^T \sigma_k.
    \end{align}
\end{proof}

\section{Proofs of Non-Positivity}
\label{app:nonpos}

The goal is to show that the two terms
\begin{equation}
\label{eq:two}
\frac{\rho_M (1-\rho_B)}{8 \tau_k^2} + 4 \gamma^2 L^2 \Theta_1 \Theta_2 + \frac{\gamma_k}{\tau_k} \left( \frac{L}{2} - \frac{1}{4 \tau_k \gamma_k} \right) \quad \textnormal{and} \quad 8 \gamma_k^2 L \Theta_1 \Theta_2 - \frac{\gamma_k (1-\tau_k)}{\tau_k}
\end{equation}
are non-positive with the parameter choices of Theorems \ref{thm:main1} and \ref{thm:main2}. We consider three cases.

\paragraph{Case 1.} Let $\gamma_k$ and $\tau_k$ be as in the statement of Theorem \ref{thm:main1}. For the first term in \eqref{eq:two},
\begin{align}
    & \frac{\rho_M (1-\rho_B)}{8 \tau_k^2} + 4 \gamma_k^2 L^2 \Theta_1 \Theta_2 + \frac{\gamma_k}{\tau_k} \left( \frac{L}{2} - \frac{1}{4 \tau_k \gamma_k} \right) \\
    = & \gamma_k^2 L^2 \left( \frac{\rho_M (1-\rho_B) c^2}{8} + 4 \Theta_1 \Theta_2 + 4 c \left( \frac{1}{2} - \frac{c}{4} \right) \right) 
\end{align}
The constraint
\begin{equation}
    c \ge \frac{2}{2 - \rho_M + \rho_B \rho_M} \left(1 + \sqrt{1 + 8 \Theta_1 \Theta_2 (2 - \rho_M + \rho_B \rho_M} ) \right)
\end{equation}
ensures that this quadratic in $c$ is non-positive. For the second term, we require $\tau_k \le 1/2$ for all $k$, which holds because $\tau_k = \frac{2}{k + \nu + 4} \le \frac{1}{2}$. Therefore,
\begin{equation}
    8 \gamma_k^2 L \Theta_1 \Theta_2 - \frac{\gamma_k (1-\tau_k)}{\tau_k} \le 8 \gamma_k^2 L \Theta_1 \Theta_2 - \frac{c L \gamma_k^2 \Theta_1 \Theta_2}{2}.
\end{equation}
The constraint $c \ge 16 \Theta_1 \Theta_2$ implies that this quantity is non-positive.

\paragraph{Case 2.} Let $\gamma$ and $\tau$ be as in the statement of Theorem \ref{thm:main2}, and suppose $\frac{1}{\sqrt{\mu L c}} \le \frac{\rho}{2 \mu}$. In this case, $\tau = \sqrt{\frac{\mu}{L c}} = \frac{1}{L c \gamma}$. As in Case 1,
\begin{align}
    & \frac{\rho_M (1-\rho_B)}{8 \tau^2} + 4 \gamma^2 L^2 \Theta_1 \Theta_2 + \frac{\gamma}{\tau} \left( \frac{L}{2} - \frac{1}{4 \tau \gamma} \right) \\
    = & \gamma^2 L^2 \left( \frac{\rho_M (1-\rho_B) c^2}{8} + 4 \Theta_1 \Theta_2 + 4 c \left( \frac{1}{2} - \frac{c}{4} \right) \right),
\end{align}
which is non-positive due to the constraints on $c$. For the second term, all we must show is that $1 - \tau \ge 1 / 2$. We have $\tau = \sqrt{\frac{\mu}{L c}} \le \frac{1}{\sqrt{c}}$, and $c$ is larger than 4, so the constraint $c \ge 16 \Theta_1 \Theta_2$ ensures that the second term in \eqref{eq:two} is non-positive.

\paragraph{Case 3.} In Theorem \ref{thm:main2}, suppose instead that $\frac{\rho}{2 \mu} \le \frac{1}{\sqrt{\mu L c}}$, so that $\gamma = \frac{\rho}{2 \mu}$ and $\tau = \frac{\rho}{2}$. This assumption implies the inequality $\frac{L}{\mu} \le \frac{4}{c \rho^2}$, so
\begin{align}
    & \frac{\rho_M (1-\rho_B)}{8 \tau^2} + 4 \gamma^2 L^2 \Theta_1 \Theta_2 + \frac{\gamma}{\tau} \left( \frac{L}{2} - \frac{1}{4 \tau \gamma} \right) \\
    & = \frac{\rho_M (1-\rho_B)}{8 \mu^2 \gamma^2} + 4 \gamma^2 L^2 \Theta_1 \Theta_2 + \frac{1}{\mu} \left( \frac{L}{2} - \frac{1}{4 \mu \gamma^2} \right) \\
    & = \frac{\rho_M (1-\rho_B)}{2 \rho^2} + \frac{\rho^2 L^2 \Theta_1 \Theta_2}{\mu^2} + \frac{L}{2 \mu} - \frac{1}{\rho^2} \\
    & \le \frac{\rho_M (1-\rho_B)}{2 \rho^2} + \frac{16 \Theta_1 \Theta_2}{c^2 \rho^2} + \frac{2}{c \rho^2} - \frac{1}{\rho^2} \\
    & = \frac{1}{c^2 \rho^2} \left( \frac{\rho_M (1-\rho_B) c^2}{2} + 16 \Theta_1 \Theta_2 + 2 c - c^2 \right).
\end{align}
This is a quadratic in $c$ with the root
\begin{equation}
    \frac{2}{2 - \rho_M + \rho_B \rho_M} \left(1 + \sqrt{1 + 8 \Theta_1 \Theta_2 (2 - \rho_M + \rho_B \rho_M} \right).
\end{equation}
Because $c$ is larger than this quantity, this term is non-positive. For the second term in \eqref{eq:two},
\begin{align}
    8 \gamma^2 L \Theta_1 \Theta_2 - \frac{\gamma (1-\tau)}{\tau} & = \frac{2 L \Theta_1 \Theta_2 \rho^2}{\mu^2} - \frac{1}{2 \mu} \le \frac{8 \Theta_1 \Theta_2}{c \mu} - \frac{1}{2 \mu} \le 0,
\end{align}
where the last inequality follows from the fact that $c \ge 16 \Theta_1 \Theta_2$.

\section{Proofs for SAGA and SVRG} 
\label{sec:saga}

Our results for SAGA and SVRG require the following lemma, which appears also as \cite[Lem. 7]{reddi}.

\begin{lemma}
\label{lem:orth}
    Suppose $X_1, \cdots, X_t$ are independent random variables satisfying $\E X_i = 0$ for all $i$. Then
    \begin{equation}
        \E \|X_1 + \cdots + X_t\|^2 = \E [ \|X_1\|^2 + \cdots + \|X_t\|^2 ].
    \end{equation}
\end{lemma}

\begin{proof}
    Our hypotheses on these random variables imply $\E [X_i X_j] = 0$ for $i \not = j$. Therefore,
    \begin{align}
        \E \|X_1 + \cdots + X_t\|^2 = \sum_{i,j = 1}^t \E [ X_i X_j ] = \|X_1\|^2 + \cdots + \|X_t\|^2.
    \end{align}
\end{proof}

We begin with a standard bound on the variance $\|\widetilde{\nabla}^{\textnormal{\tiny SAGA}}_{k+1} - \nabla f(x_{k+1})\|^2$ that is an easy consequence of the variance bound in \cite{SAGA}, but \cite{SAGA} and related works \cite{pointSAGA,katyusha,svrg,proxsvrg} ultimately use a looser bound in their convergence analysis.
\begin{lemma}
\label{lem:firstvarbound}
The variance of the SAGA gradient estimator with minibatches of size $b$ is bounded as follows:
\begin{equation}
    \E \|\widetilde{\nabla}^{\textnormal{\tiny SAGA}}_{k+1} - \nabla f(x_{k+1}) \|^2 \le \frac{1}{b n} \sum_{i=1}^n \| \nabla f_i(x_{k+1}) - \nabla f_i(\varphi_k^i) \|^2
\end{equation}
\end{lemma}

\begin{proof}
Let $X_i = \nabla f_j (x_{k+1}) - \nabla f_j(\varphi_k^j)$ for $j \in J_k$. We then have
\begin{align}
    &\E \|\widetilde{\nabla}^{\textnormal{\tiny SAGA}}_{k+1} - \nabla f(x_{k+1}) \|^2 \notag \\
    = & \E \| \frac{1}{b} \sum_{j \in J_k} (\nabla f_j(x_{k+1}) - \nabla f_j(\varphi_k^j)) + \frac{1}{n} \sum_{i=1}^n \nabla f_i(\varphi_k^i) - \nabla f(x_{k+1}) \|^2 \\
    = & \E \| \frac{1}{b} \sum_{i = 1}^b X_i - \E X_i \|^2 \\
    \symnum{1}{=} & \frac{1}{b^2} \E \sum_{i = 1}^b \left\| X_i \right\|^2 \\
    = & \frac{1}{b^2} \E \sum_{j \in J_k} \left\| \nabla f_j(x_{k+1}) - \nabla f_j(\varphi_k^j) \right\|^2 \\
    = & \frac{1}{b n} \sum_{i=1}^n \| \nabla f_i(x_{k+1}) - \nabla f_i(\varphi_k^i) \|^2.
\end{align}
Equality \numcirc{1} is due to Lemma \ref{lem:orth}.
\end{proof}

Lemma \ref{lem:firstvarbound} provides a variance bound that is compatible with the MSEB property, as we show in the following lemma.

\begin{lemma}
\label{lem:varbound}
The SAGA gradient estimator satisfies the MSEB property with $M_1 = \frac{3 n}{b^2}$, $\rho_M = \frac{b}{2 n}$, $M_2 = 0$, and $\rho_B = \rho_F = 1$.
\end{lemma}
\begin{proof}
Lemma \ref{lem:firstvarbound} shows that the MSE of the SAGA gradient estimator is dominated by $\frac{1}{b n} \sum_{i=1}^n \mathbb{E} \| \nabla f_i(x_{k+1})$ $- \nabla f_i(\varphi_k^i)\|^2$, so we choose this sequence for $\mathcal{M}_k$. Using the inequality $\|a-c\|^2 \le (1+\frac{2 n}{b}) \|a-b\|^2 + (1+\frac{b}{2 n}) \|b-c\|^2$,
\begin{align}
\label{eq:mainsagavar}
    \mathcal{M}_k = & \frac{1}{b n} \sum_{i=1}^n \mathbb{E} \| \nabla f_i(x_{k+1}) - \nabla f_i(\varphi_k^i) \|^2 \\
    %%%%%%%%%%%%%%%%%%%%%%%%%%%%%%%%%%%%%%%%%%%%%%%%%%
    \le & \frac{1+\frac{2 n}{b}}{b n} \sum_{i=1}^n \mathbb{E} \| \nabla f_i(x_{k+1}) - \nabla f_i(x_k) \|^2 + \frac{1+\frac{b}{2 n}}{b n} \sum_{i=1}^n \mathbb{E} \| \nabla f_i(x_k) - \nabla f_i(\varphi_k^i) \|^2 \\
    %%%%%%%%%%%%%%%%%%%%%%%%%%%%%%%%%%%%%%%%%%%%%%%%%%%%%%%%%
    \symnum{1}{=} & \frac{1+\frac{2 n}{b}}{b n} \sum_{i=1}^n \mathbb{E} \| \nabla f_i(x_{k+1}) - \nabla f_i(x_k) \|^2 + \frac{1+\frac{b}{2n}}{b n} \left( 1 - \frac{b}{n} \right) \sum_{i=1}^n \mathbb{E} \| \nabla f_i(x_k) - \nabla f_i(\varphi_{k-1}^i)\|^2 \\
    %%%%%%%%%%%%%%%%%%%%%%%%%%%%%%%%%%%%%%%%%%%%%%%%%%%%%%%%%
    \symnum{2}{\le} & \frac{3}{b^2} \sum_{i=1}^n \mathbb{E} \| \nabla f_i(x_{k+1}) - \nabla f_i(x_k) \|^2 + \frac{1}{b n} \left( 1 - \frac{b}{2 n} \right) \sum_{i=1}^n \mathbb{E} \| \nabla f_i(x_k) - \nabla f_i(\varphi_{k-1}^i)\|^2 \\
    %%%%%%%%%%%%%%%%%%%%%%%%%%%%%%%%%%%%%%%%%%%%%%%%%%%%%%%%%
    = & \frac{3}{b^2} \sum_{i=1}^n \mathbb{E} \| \nabla f_i(x_{k+1}) - \nabla f_i(x_k) \|^2 + \left( 1 - \frac{b}{2 n} \right) \mathcal{M}_{k-1}.
    %%%%%%%%%%%%%%%%%%%%%%%%%%%%%%%%%%%%%%%%%%%%%%%%%%%%%%%%%
\end{align}
Equality \numcirc{1} follows from computing expectations and the update rule for $\varphi_k^i$:
\begin{align}
    \sum_{i=1}^n \mathbb{E} \| \nabla f_i(x_k) - \nabla f_i(\varphi_k^i)\|^2 & = \frac{1}{b} \sum_{j \in J_{k-1}} \mathbb{E} \| \nabla f_j(x_k) - \nabla f_j (\varphi_k^j)\|^2 \notag \\
    & \quad \quad + \mathbb{E} \sum_{i \not \in J_{k-1} }^n \| \nabla f_i(x_k) - \nabla f_i(\varphi_{k-1}^i)\|^2 \\
    &= 0 + \left(1-\frac{b}{n}\right) \sum_{i = 1}^n \mathbb{E} \| \nabla f_i(x_k) - \nabla f_i(\varphi_{k-1}^i)\|^2,
\end{align}
and \numcirc{2} follows from the the inequalities $\left( 1 + \frac{b}{2 n} \right) \left( 1 - \frac{b}{n} \right) \le \left( 1 - \frac{b}{2 n} \right)$ and $1+\frac{2 n}{b} \le \frac{3 n}{b}$. This shows that we can take $M_1 = \frac{3 n}{b^2}$, $M_2 = 0$, and $\rho_F = 1$. Because the SAGA gradient estimator is unbiased, we can clearly set $\rho_B = 1$, proving the claim.
\end{proof}

A similar result holds for the SVRG gradient estimator.

\begin{corollary}
\label{lem:svrgvarbound}
The SVRG gradient estimator satisfies the MSEB property with $M_1 = \frac{3 p}{b}$, $\rho_M = \frac{1}{2 p}$, $M_2 = 0$, and $\rho_B = \rho_F = 1$.
\end{corollary}
\begin{proof}
Following the same argument as in the proof of Lemma \ref{lem:firstvarbound}, we have the bound
\begin{equation}
    \mathbb{E} \|\widetilde{\nabla}^{\textnormal{\tiny SVRG}}_{k+1} - \nabla f(x_{k+1}) \|^2 \le \frac{1 - 1/p}{b n} \sum_{i=1}^n \mathbb{E} \| \nabla f_i(x_{k+1}) - \nabla f_i(\widetilde{x}) \|^2.
\end{equation}
The factor $1 - 1/p$ that appears is due to the fact that $\widetilde{\nabla}_{k+1} = \nabla f(x_{k+1})$ with probability $1/p$. With $\mathcal{M}_k = \frac{1 - 1/p}{b n} \sum_{i=1}^n \mathbb{E} \| \nabla f_i(x_{k+1}) - \nabla f_i(\widetilde{x}) \|^2$, we follow the proof of Lemma \ref{lem:varbound}.
\begin{align}
\label{eq:mainsvrgvar}
    \mathcal{M}_k = & \frac{1 - 1/p}{b n} \sum_{i=1}^n \mathbb{E} \| \nabla f_i(x_{k+1}) - \nabla f_i(\widetilde{x}) \|^2 \\
    \le & \frac{( 1 + 2 p ) ( 1 - 1/p )}{b n} \sum_{i=1}^n \mathbb{E} \| \nabla f_i(x_{k+1}) - \nabla f_i(x_k) \|^2 + \frac{(1+\frac{1}{2 p})(1 - 1/p)}{b n} \sum_{i=1}^n \mathbb{E} \| \nabla f_i(x_k) - \nabla f_i(\widetilde{x}) \|^2 \\
    %%%%%%%%%%%%%%%%%%%%%%%%%%%%%%%%%%%%%%%%%%%%%%%%%%%%%%%%%
    \symnum{1}{=} & \frac{(1+2 p)(1 - 1/p)}{b n} \sum_{i=1}^n \mathbb{E} \| \nabla f_i(x_{k+1}) - \nabla f_i(x_k) \|^2 + \frac{(1+\frac{1}{2 p})(1 - 1/p)^2}{b n} \sum_{i=1}^n \mathbb{E} \| \nabla f_i(x_k) - \nabla f_i(\widetilde{x})\|^2 \\
    %%%%%%%%%%%%%%%%%%%%%%%%%%%%%%%%%%%%%%%%%%%%%%%%%%%%%%%%%
    \le & \frac{3 p}{b n} \sum_{i=1}^n \mathbb{E} \| \nabla f_i(x_{k+1}) - \nabla f_i(x_k) \|^2 + \left( 1 - \frac{1}{2 p} \right) \mathcal{M}_{k-1}.
\end{align}
Equality \numcirc{1} follows from the fact that $\widetilde{x} = x_k$ with probability $1 / p$.
\end{proof}

With the MSEB property established for the SAGA and SVRG gradient estimators, we can apply Theorems \ref{thm:main1} and \ref{thm:main2} to get a rate of convergence. For the SAGA estimator, Lemma \ref{lem:varbound} ensures that the choices $c = \frac{96 n^2}{b^3}$ and $\rho = \frac{b}{2 n}$ satisfy the hypotheses of Theorems \ref{thm:main1} and \ref{thm:main2} as long as $b \le 4 \sqrt{2} n^{2/3}$. Similarly, for the SVRG estimator, the choices $c = \frac{b}{96 p^2}$ and $\rho = \frac{1}{2 p}$ satisfy the conditions of Theorems \ref{thm:main1} and \ref{thm:main2} as long as $b \le 32 p^2$.

\section{Proofs for SARAH}
\label{sec:sarah}

To prove the convergence rates of Theorem \ref{thm:sarah}, we first show that the SARAH gradient estimator satisfies the MSEB property.

\begin{lemma}
\label{lem:sarahmseb}
    The SARAH gradient estimator satisfies the MSEB property with $M_1 = 1$, $M_2 = 0$, $\rho_M = 1 / p$, $\rho_B = 1 / p$, and $\rho_F = 1$.
\end{lemma}

\begin{proof}
 
The SARAH gradient estimator is equal to $\nabla f(x_{k+1})$ with probability $1 / p$, so the expectation of the SARAH gradient estimator is
\begin{align}
    \E \widetilde{\nabla}^{\textnormal{\tiny SARAH}}_{k+1} &= \frac{1}{p} \nabla f(x_{k+1}) + \left(1 - \frac{1}{p}\right) \left( \frac{1}{b} \E \left(\sum_{j \in J_k} \nabla f_j(x_{k+1}) - \nabla f_j(x_k) \right) + \widetilde{\nabla}^{\textnormal{\tiny SARAH}}_k \right)
    \\
    & = \frac{1}{p} \nabla f(x_{k+1}) + \left( 1 - \frac{1}{p} \right) \left( \nabla f(x_{k+1}) - \nabla f(x_k) + \widetilde{\nabla}^{\textnormal{\tiny SARAH}}_k \right)
\end{align}
Therefore,
\begin{align}
    \nabla f(x_{k+1}) - \E \widetilde{\nabla}^{\textnormal{\tiny SARAH}}_{k+1} = \left(1-\frac{1}{p}\right) \left( \nabla f(x_k) - \widetilde{\nabla}^{\textnormal{\tiny SARAH}}_k \right),
\end{align}
so $\rho_B = 1 / p$. Next, we prove a bound on the MSE. Let $\mathbb{E}_{k,p}$ denote the expectation conditioned on the first $k$ iterations and the event that the full gradient is not computed at iteration $k+1$. Under the condition that the full gradient is not computed, the expectation of the SARAH estimator is
\begin{align}
    \mathbb{E}_{k,p} \widetilde{\nabla}^{\textnormal{\tiny SARAH}}_{k+1} &= \frac{1}{b} \mathbb{E}_{k,p} \left(\sum_{j \in J_k} \nabla f_j(x_{k+1}) - \nabla f_j(x_k) \right) + \widetilde{\nabla}^{\textnormal{\tiny SARAH}}_k
    \\
    & = \nabla f(x_{k+1}) - \nabla f(x_k) + \widetilde{\nabla}^{\textnormal{\tiny SARAH}}_k
\end{align}

The beginning of our proof is similar to the proof of the MSE bound in \cite[Lem. 2]{sarah}.
\begin{align}
    & \mathbb{E}_{k,p} \|\widetilde{\nabla}^{\textnormal{\tiny SARAH}}_{k+1} - \nabla f(x_{k+1})\|^2 \notag \\
    = & \mathbb{E}_{k,p} \left\| \widetilde{\nabla}^{\textnormal{\tiny SARAH}}_k - \nabla f(x_k) + \nabla f(x_k) - \nabla f(x_{k+1}) + \widetilde{\nabla}^{\textnormal{\tiny SARAH}}_{k+1} - \widetilde{\nabla}^{\textnormal{\tiny SARAH}}_k \right\|^2 \\
    = & \left\| \widetilde{\nabla}^{\textnormal{\tiny SARAH}}_k - \nabla f(x_k) \right\|^2 + \left\|\nabla f(x_k) - \nabla f(x_{k+1}) \right\|^2 + \mathbb{E}_{k,p} \left\| \widetilde{\nabla}^{\textnormal{\tiny SARAH}}_{k+1} - \widetilde{\nabla}^{\textnormal{\tiny SARAH}}_k \right\|^2 \notag \\
    & + 2 \langle \nabla f(x_k ) - \widetilde{\nabla}^{\textnormal{\tiny SARAH}}_k, \nabla f(x_{k+1}) - \nabla f(x_k) \rangle \\
    & \quad - 2 \left\langle \nabla f(x_k ) - \widetilde{\nabla}^{\textnormal{\tiny SARAH}}_k, \mathbb{E}_{k,p} \left[ \widetilde{\nabla}^{\textnormal{\tiny SARAH}}_{k+1} - \widetilde{\nabla}^{\textnormal{\tiny SARAH}}_k \right] \right\rangle \\
    & \quad \quad - 2 \left\langle \nabla f(x_{k+1}) - \nabla f(x_k), \mathbb{E}_{k,p} \left[ \widetilde{\nabla}^{\textnormal{\tiny SARAH}}_{k+1} - \widetilde{\nabla}^{\textnormal{\tiny SARAH}}_k \right] \right\rangle.
\end{align}
We consider each inner product separately. The first inner product is equal to
\begin{align}
    & 2 \langle \nabla f(x_k ) - \widetilde{\nabla}^{\textnormal{\tiny SARAH}}_k, \nabla f(x_{k+1}) - \nabla f(x_k) \rangle \notag \\
    = & - \|\nabla f(x_k ) - \widetilde{\nabla}^{\textnormal{\tiny SARAH}}_k\|^2 - \|\nabla f(x_{k+1}) - \nabla f(x_k)\|^2 + \|\nabla f(x_{k+1}) - \widetilde{\nabla}^{\textnormal{\tiny SARAH}}_k\|^2.
\end{align}
For the next two inner products, we use the fact that
\begin{align}
    & \mathbb{E}_{k,p}[\widetilde{\nabla}^{\textnormal{\tiny SARAH}}_{k+1} - \widetilde{\nabla}^{\textnormal{\tiny SARAH}}_k] = \nabla f(x_{k+1}) - \nabla f(x_k).
\end{align}
With this equality established, we see that the second inner product is equal to
\begin{align}
    & - 2 \left\langle \nabla f(x_k) - \widetilde{\nabla}^{\textnormal{\tiny SARAH}}_k, \mathbb{E}_{k,p} \left[ \widetilde{\nabla}^{\textnormal{\tiny SARAH}}_{k+1} - \widetilde{\nabla}^{\textnormal{\tiny SARAH}}_k \right] \right\rangle \notag \\
    %%%%%%%%%%%%%%%%%%%%%%%%%%%%%%%%%%%%%%%%%%%%%%%%%%%%%%%%%%%%%%%
    = & - 2 \langle \nabla f(x_k) - \widetilde{\nabla}^{\textnormal{\tiny SARAH}}_k, \nabla f(x_{k+1}) - \nabla f(x_k) \rangle \\
    =& \| \nabla f(x_k ) - \widetilde{\nabla}^{\textnormal{\tiny SARAH}}_k \|^2 + \| \nabla f(x_{k+1}) - \nabla f(x_k) \|^2 - \|\nabla f(x_{k+1}) - \widetilde{\nabla}^{\textnormal{\tiny SARAH}}_k \|^2.
\end{align}
The third inner product can be bounded using a similar procedure.
\begin{align}
    & - 2 \left\langle \nabla f(x_{k+1}) - \nabla f(x_k), \mathbb{E}_{k,p} \left[ \widetilde{\nabla}^{\textnormal{\tiny SARAH}}_{k+1} - \widetilde{\nabla}^{\textnormal{\tiny SARAH}}_k \right] \right\rangle \notag \\
    %%%%%%%%%%%%%%%%%%%%%%%%%%%%%%%%%%%%%%%%%%%%%%%%%%%%
    = & - 2 \langle \nabla f(x_{k+1}) - \nabla f(x_k), \nabla f(x_{k+1}) - \nabla f(x_k) \rangle \\
    %%%%%%%%%%%%%%%%%%%%%%%%%%%%%%%%%%%%%%%%%%%%%%%%%%%%%
    = & - 2 \| \nabla f(x_{k+1}) - \nabla f(x_k) \|^2.
\end{align}
Altogether, we have
\begin{align}
    & \mathbb{E}_{k,p} \|\widetilde{\nabla}^{\textnormal{\tiny SARAH}}_{k+1} - \nabla f(x_{k+1})\|^2 \\
    %%%%%%%%%%%%%%%%%%%%%%%%%%%%%%%%%%%%%%%%%%%%%%%%%%%%%%
    \le & \left\| \widetilde{\nabla}^{\textnormal{\tiny SARAH}}_k - \nabla f(x_k) \right\|^2 - \|\nabla f(x_{k+1}) - \nabla f(x_k)\|^2 + \mathbb{E}_{k,p} \| \widetilde{\nabla}^{\textnormal{\tiny SARAH}}_{k+1} - \widetilde{\nabla}^{\textnormal{\tiny SARAH}}_k\|^2 \\
    %%%%%%%%%%%%%%%%%%%%%%%%%%%%%%%%%%%%%%%%%%%%%%%%%%%%%%
    \le & \left\| \widetilde{\nabla}^{\textnormal{\tiny SARAH}}_k - \nabla f(x_k) \right\|^2 + \mathbb{E}_{k,p} \| \widetilde{\nabla}^{\textnormal{\tiny SARAH}}_{k+1} - \widetilde{\nabla}^{\textnormal{\tiny SARAH}}_k\|^2.
\end{align}

For the second term,
\begin{align}
    \mathbb{E}_{k,p} \| \widetilde{\nabla}^{\textnormal{\tiny SARAH}}_{k+1} - \widetilde{\nabla}^{\textnormal{\tiny SARAH}}_k\|^2 = & \mathbb{E}_{k,p} \left\| \frac{1}{b} \left(\sum_{j \in J_k} \nabla f_j(x_{k+1}) - \nabla f_j(x_k) \right) \right\|^2\\
    \le & \frac{1}{b} \mathbb{E}_{k,p} \left[\sum_{j \in J_k} \| \nabla f_j(x_{k+1}) - \nabla f_j(x_k)\|^2 \right] \\
    = & \frac{1}{n} \sum_{i=1}^n \| \nabla f_i(x_{k+1}) - \nabla f_i(x_k)\|^2.
\end{align}
The inequality is Jensen's. This results in the recursive inequality
\begin{align}
    & \mathbb{E}_{k,p} \|\widetilde{\nabla}^{\textnormal{\tiny SARAH}}_{k+1} - \nabla f(x_{k+1})\|^2 \notag \\
    \le & \left\| \widetilde{\nabla}^{\textnormal{\tiny SARAH}}_k - \nabla f(x_k) \right\|^2 + \frac{1}{n} \sum_{i=1}^n \| \nabla f_i(x_{k+1}) - \nabla f_i(x_k)\|^2.
\end{align}
This provides a bound on the MSE under the condition that the full gradient is not computed at iteration $k$. If the full gradient is computed, the MSE of the estimator is clearly equal to zero, so applying the full expectation operator yields
\begin{align}
    & \mathbb{E} \|\widetilde{\nabla}^{\textnormal{\tiny SARAH}}_{k+1} - \nabla f(x_{k+1})\|^2 \notag \\
    \le & \left(1 - \frac{1}{p} \right) \left( \mathbb{E} \left\| \widetilde{\nabla}^{\textnormal{\tiny SARAH}}_k - \nabla f(x_k) \right\|^2 + \frac{1}{n} \sum_{i=1}^n \mathbb{E} \| \nabla f_i(x_{k+1}) - \nabla f_i(x_k)\|^2 \right).
\end{align}
With $\mathcal{M}_k = \mathbb{E} \|\widetilde{\nabla}^{\textnormal{\tiny SARAH}}_{k+1} - \nabla f(x_{k+1})\|^2$, it is clear that we can take $M_1 = 1, \rho_M = 1 / p$, $M_2 = 0$, and $\rho_F = 1$.
\end{proof}

With these MSEB constants established, convergence rates easily follow from Theorems \ref{thm:main1} and \ref{thm:main2} with $c = 144 p^4$ and $\rho = 1 / p$.

\section{Proofs for SARGE}
\label{sec:sarge}

For the proofs in this section, we rewrite the SARGE gradient estimator in terms of the SAGA estimator to make the analysis easier to follow. Define the operator
\begin{equation}
    \widetilde{\nabla}^{\xi\textnormal{\tiny-SAGA}}_{k+1} \defeq \frac{1}{b} \left(\sum_{j \in J_k} \nabla f_j(x_k) - \nabla f_j(\xi^j_k) \right) + \frac{1}{n} \sum_{i=1}^n \nabla f_i(\xi^i_k),
\end{equation}
where the variables $\{\xi_k^i\}_{i=1}^n$ follow the update rules $\xi_{k+1}^j = x_k$ for all $j \in J_k$ and $\xi_{k+1}^i = \xi_k^i$ for all $i \not \in J_k$. The SARGE estimator is equal to
\begin{equation}
    \widetilde{\nabla}^{\textnormal{\tiny SARGE}}_{k+1} = \widetilde{\nabla}^{\textnormal{\tiny SAGA}}_{k+1} - \left(1 - \frac{b}{n} \right) \left( \widetilde{\nabla}^{\xi\textnormal{\tiny-SAGA}}_{k+1} - \widetilde{\nabla}^{\textnormal{\tiny SARGE}}_{k+1} \right).
\end{equation}
Before we begin, we require a bound on the MSE of the $\xi$-SAGA gradient estimator that follows immediately from Lemma \ref{lem:varbound}.

\begin{lemma}
\label{lem:sagaforsarge}
    The MSE of the $\xi$-SAGA gradient estimator satisfies the following bound:
    \begin{equation}
        \mathbb{E} \left\|\widetilde{\nabla}^{\xi\textnormal{\tiny-SAGA}}_{k+1} - \nabla f(x_k) \right\|^2 \le \frac{3}{b^2} \sum_{\ell = 1}^k \left(1 - \frac{b}{2 n} \right)^{k - \ell} \sum_{i=1}^n \mathbb{E} \|\nabla f_i(x_{\ell}) - \nabla f_i(x_{\ell-1}) \|^2.
    \end{equation}
\end{lemma}

\begin{proof}
Following the proof of Lemma \ref{lem:firstvarbound},
\begin{align}
    & \E\left\|\widetilde{\nabla}^{\xi\textnormal{\tiny-SAGA}}_{k+1} - \nabla f(x_k) \right\|^2 \notag \\
    = & \E \left\|\frac{1}{b} \sum_{j \in J_k} \left(\nabla f_j(x_k) - \nabla f_j(\xi_k^j)\right) - \nabla f(x_k) + \frac{1}{n} \sum_{i=1}^n \nabla f_i(\xi_k^i) \right\|^2 \\
    %%%%%%%%%%%%%%%%%%%%%%%%%%%%%%%%%%%%%%%
    % \symnum{1}{\le} & \frac{1}{b n} \sum_{i=1}^n \left\|\nabla f_i(x_k) - \nabla f_i(\xi_k^i) \right\|^2 - \left\| \nabla f(x_k) - \frac{1}{n} \sum_{i=1}^n \nabla f_i(\xi_k^i) \right\|^2 \\
    %%%%%%%%%%%%%%%%%%%%%%%%%%%%%%%%%%%%%%%%%%%%
    \symnum{1}{=} & \frac{1}{b n} \sum_{i=1}^n \left\|\nabla f_i(x_k) - \nabla f_i(\xi_k^i) \right\|^2.
\end{align}
Equality \numcirc{1} is an application of Lemma \ref{lem:orth}. To continue, we follow the proof of Lemma \ref{lem:varbound}.
\begin{align}
    & \mathbb{E} \left\|\widetilde{\nabla}^{\xi\textnormal{\tiny-SAGA}}_{k+1} - \nabla f(x_k) \right\|^2 \\
    %%%%%%%%%%%%%%%%%%%%%%%%%%%%%%%%%%%%%%%%%%%%%%%%%
    \le & \frac{1}{b n} \sum_{i=1}^n \mathbb{E} \left\|\nabla f_i(x_k) - \nabla f_i(\xi_k^i) \right\|^2 \\
    %%%%%%%%%%%%%%%%%%%%%%%%%%%%%%%%%%%%%%%%%%%%%%%%%%%%%%%
    \le & \frac{(1 + \frac{2 n}{b})}{b n} \sum_{i=1}^n\mathbb{E} \|\nabla f_i(x_k) - \nabla f_i(x_{k-1}) \|^2 + \frac{1}{b n} \left( 1 + \frac{b}{2 n} \right) \sum_{i=1}^n \mathbb{E} \left\| \nabla f_i(x_{k-1}) - \nabla f_i(\xi_k^i) \right\|^2 \\
    %%%%%%%%%%%%%%%%%%%%%%%%%%%%%%%%%%%%%%%%%%%%%%%%%%%%%%%%%
    \symnum{2}{=} & \frac{(1 + \frac{2 n}{b})}{b n} \sum_{i=1}^n\mathbb{E} \|\nabla f_i(x_k) - \nabla f_i(x_{k-1}) \|^2 + \frac{1}{b n} \left( 1 + \frac{b}{2 n} \right) \left(1 - \frac{b}{n}\right) \sum_{i=1}^n \mathbb{E} \left\| \nabla f_i(x_{k-1}) - \nabla f_i(\xi_{k-1}^i) \right\|^2 \\
    %%%%%%%%%%%%%%%%%%%%%%%%%%%%%%%%%%%%%%%%%%%%%%%%%%%%
    \symnum{3}{\le} & \frac{3}{b^2} \sum_{i=1}^n\mathbb{E} \|\nabla f_i(x_k) - \nabla f_i(x_{k-1}) \|^2 + \frac{1}{b n} \left( 1 - \frac{b}{2 n} \right) \sum_{i=1}^n \mathbb{E} \left\| \nabla f_i(x_{k-1}) - \nabla f_i(\xi_{k-1}^i) \right\|^2 \\
    %%%%%%%%%%%%%%%%%%%%%%%%%%%%%%%%%%%%%%%%%%%%%%%%
    \le & \frac{3}{b^2} \sum_{\ell = 1}^k \left(1 - \frac{b}{2 n} \right)^{k - \ell} \sum_{i=1}^n \mathbb{E} \|\nabla f_i(x_{\ell}) - \nabla f_i(x_{\ell-1}) \|^2.
\end{align}
Equality \numcirc{2} follows from computing expectations, 
and \numcirc{3} uses the estimate $\left(1-\frac{b}{n}\right)\left(1+\frac{b}{2n}\right) \le \left(1-\frac{b}{2n}\right)$.
\end{proof}

Due to the recursive nature of the SARGE gradient estimator, its MSE depends on the difference between the current estimate and the estimate from the previous iteration. This is true for the recursive SARAH gradient estimate as well, but bounding the quantity $\|\widetilde{\nabla}^{\textnormal{\tiny SARAH}}_k - \widetilde{\nabla}^{\textnormal{\tiny SARAH}}_{k-1}\|^2$ is a much more straightforward task than bounding the same quantity for the SARGE estimator. The next lemma provides this bound.

\begin{lemma}
\label{lem:vbound}
    The SARGE gradient estimator satisfies the following bound:
    \begin{align}
        & \mathbb{E} \| \widetilde{\nabla}^{\textnormal{\tiny SARGE}}_{k+1} - \widetilde{\nabla}^{\textnormal{\tiny SARGE}}_k\|^2 \notag \\
        \le & \frac{27 + 12 b}{n^2} \sum_{\ell=1}^k \left( 1 - \frac{b}{2 n} \right)^{k - \ell} \sum_{i=1}^n \mathbb{E} \|\nabla f_i(x_\ell) - \nabla f_i(x_{\ell - 1}) \|^2 \notag \\
        & + \frac{12}{n} \sum_{i=1}^n \mathbb{E} \|\nabla f_i(x_{k+1}) - \nabla f_i(x_k) \|^2 + \frac{3 b^2}{2 n^2} \mathbb{E} \left\|\nabla f(x_k) - \widetilde{\nabla}^{\textnormal{\tiny SARGE}}_k \right\|^2.
    \end{align}
\end{lemma}

\begin{proof}
    To begin, we use the standard inequality $\|a - c\|^2 \le (1+\delta) \|a - b\|^2 + (1+\delta^{-1}) \|b - c\|^2$ for any $\delta > 0$ twice. For simplicity, we set $\delta = \sqrt{3/2} - 1$ and use the fact that $1 + \frac{1}{\sqrt{3/2} - 1} \le 6$ for both applications of this inequality.
    \begin{align}
    \label{eq:vbound}
        & \mathbb{E} \| \widetilde{\nabla}^{\textnormal{\tiny SARGE}}_{k+1} - \widetilde{\nabla}^{\textnormal{\tiny SARGE}}_k\|^2 \notag \\
        = & \mathbb{E} \left\| \widetilde{\nabla}^{\textnormal{\tiny SAGA}}_{k+1} - \left(1-\frac{b}{n}\right) \left( \widetilde{\nabla}^{\xi\textnormal{\tiny-SAGA}}_{k+1} - \widetilde{\nabla}^{\textnormal{\tiny SARGE}}_k \right) - \widetilde{\nabla}^{\textnormal{\tiny SARGE}}_k \right\|^2 \\
        %%%%%%%%%%%%%%%%%%%%%%%%%%%%%%%%%%%%%%%%%%%%%%
        \le & 6 \mathbb{E} \left\| \widetilde{\nabla}^{\textnormal{\tiny SAGA}}_{k+1} - \widetilde{\nabla}^{\xi\textnormal{\tiny-SAGA}}_{k+1} \right\|^2 + \frac{\sqrt{3} b^2}{\sqrt{2} n^2} \mathbb{E} \left\|\widetilde{\nabla}^{\xi\textnormal{\tiny-SAGA}}_{k+1} - \widetilde{\nabla}^{\textnormal{\tiny SARGE}}_k \right\|^2 \\
        %%%%%%%%%%%%%%%%%%%%%%%%%%%%%%%%%%%%%%%%%%%%%%%%%%
        \le & 6 \mathbb{E} \left\| \widetilde{\nabla}^{\textnormal{\tiny SAGA}}_{k+1} - \widetilde{\nabla}^{\xi\textnormal{\tiny-SAGA}}_{k+1} \right\|^2 + \frac{6 \sqrt{3} b^2}{\sqrt{2} n^2} \mathbb{E} \left\|\widetilde{\nabla}^{\xi\textnormal{\tiny-SAGA}}_{k+1} - \nabla f(x_k) \right\|^2 + \frac{3 b^2}{2 n^2} \mathbb{E} \left\|\nabla f(x_k) - \widetilde{\nabla}^{\textnormal{\tiny SARGE}}_k \right\|^2 \\
        %%%%%%%%%%%%%%%%%%%%%%%%%%%%%%%%%%%%%%%%%%%%%%%%%
        \le & 6 \mathbb{E} \left\| \widetilde{\nabla}^{\textnormal{\tiny SAGA}}_{k+1} - \widetilde{\nabla}^{\xi\textnormal{\tiny-SAGA}}_{k+1} \right\|^2 + \frac{9 b^2}{n^2} \mathbb{E} \left\|\widetilde{\nabla}^{\xi\textnormal{\tiny-SAGA}}_{k+1} - \nabla f(x_k) \right\|^2 + \frac{3 b^2}{2 n^2} \mathbb{E} \left\|\nabla f(x_k) - \widetilde{\nabla}^{\textnormal{\tiny SARGE}}_k \right\|^2.
    \end{align}
    We now bound the first two of these three terms separately. Consider the first term.
    \begin{align}
        & 6 \mathbb{E} \left\| \widetilde{\nabla}^{\textnormal{\tiny SAGA}}_{k+1} - \widetilde{\nabla}^{\xi\textnormal{\tiny-SAGA}}_{k+1} \right\|^2 \\
        %%%%%%%%%%%%%%%%%%%%%%%%%%%%%%%%%%%%%%%%%%%%%%%%%%%%%%%%%%%%%%
        = & 6 \mathbb{E} \Bigg\| \frac{1}{b} \left( \sum_{j \in J_k} \nabla f_j(x_{k+1}) - \nabla f_j(\varphi_k^j) \right) + \frac{1}{n} \sum_{i=1}^n \nabla f_i(\varphi_k^i) \notag \\
        & - \frac{1}{b} \left( \sum_{j \in J_{k-1}} \nabla f_j(x_k) - \nabla f_j(\xi_k^j) \right) - \frac{1}{n} \sum_{i=1}^n \nabla f_i(\xi_k^i) \Bigg\|^2 \\
        %%%%%%%%%%%%%%%%%%%%%%%%%%%%%%%%%%%%%%%%%%%%%%%%%%%%%%%%%%%%%%%
        \le & 12 \mathbb{E} \Bigg\| \frac{1}{b} \left( \sum_{j \in J_k} \nabla f_j(x_{k+1}) - \nabla f_j(x_k) \right) \Bigg\|^2 \notag \\
        & + 12 \mathbb{E} \Bigg\| \frac{1}{b} \left(\sum_{j \in J_k} \nabla f_j(\varphi_k^j) - \nabla f_j(\xi_k^j) \right) - \frac{1}{n} \sum_{i=1}^n \nabla f_i(\varphi_k^i) + \frac{1}{n} \sum_{i=1}^n \nabla f_i(\xi_k^i) \Bigg\|^2 \\
        %%%%%%%%%%%%%%%%%%%%%%%%%%%%%%%%%%%%%%%%%%%%%%%%%%%%
        \symnum{1}{=} & 12 \mathbb{E} \Bigg\| \frac{1}{b} \left( \sum_{j \in J_k} \nabla f_j(x_{k+1}) - \nabla f_j(x_k) \right) \Bigg\|^2 + 12 \mathbb{E} \left\| \frac{1}{b} \left( \sum_{j \in J_k} \nabla f_j(\varphi_k^j) - \nabla f_j(\xi_k^j) \right) \right\|^2 \notag \\
        & - 12 \left\| \frac{1}{n} \sum_{i=1}^n \nabla f_i(\varphi_k^i) + \frac{1}{n} \sum_{i=1}^n \nabla f_i(\xi_k^i) \right\|^2 \\
        %%%%%%%%%%%%%%%%%%%%%%%%%%%%%%%%%%%%%%%%%%%%%%%%%%
        \le & \frac{12}{n} \sum_{i=1}^n \mathbb{E} \left\| \nabla f_i(x_{k+1}) - \nabla f_i(x_k) \right\|^2 + 12 \mathbb{E} \left\| \frac{1}{b} \left( \sum_{j \in J_k} \nabla f_j(\varphi_k^j) - \nabla f_j (\xi_k^j) \right) \right\|^2 \\ 
        %%%%%%%%%%%%%%%%%%%%%%%%%%%%%%%%%%%%%%%%%%%%%%%%%%
        \le & \frac{12}{n} \sum_{i=1}^n \mathbb{E} \left\| \nabla f_i(x_{k+1}) - \nabla f_i(x_k) \right\|^2 + \frac{12}{b} \mathbb{E} \sum_{j \in J_k} \left\| \nabla f_j(\varphi_k^j) - \nabla f_j (\xi_k^j) \right\|^2.
    \end{align}
    Equality \numcirc{1} is the standard variance decomposition, which states that for any random variable $X$, $\E \| X - \E X \|^2 = \E \|X\|^2 - \|\E X\|^2$. The second term can be reduced further by computing the expectation. Let $j_k$ be any element of $J_k$. The probability that $\nabla f_{j_k} (\varphi_k^{j_k}) = \nabla f_{j_{k-1}}(x_k)$ is equal to the probability that $j_k \in J_{k-1}$, which is $b / n$. The probability that $\nabla f_{j_k}(\varphi_k^{j_k}) = \nabla f_{j_{k-2}} (x_{k-1})$ is equal to the probability that $j_k \not \in J_{k-1}$ and $j_k \in J_{k-2}$, which is $b / n \left( 1 - b / n \right)$. Continuing in this way,
    \begin{align}
        & \mathbb{E} \left\|\nabla f_{j_k}(\varphi_k^{j_k}) - \nabla f_{j_k}(\xi_k^{j_k}) \right\|^2 = \frac{b}{n} \sum_{\ell = 1}^k \left(1 - \frac{b}{n}\right)^{k - \ell} \mathbb{E} \| \nabla f_{j_{\ell-1}}(x_{\ell}) - \nabla f_{j_{\ell-1}}(x_{\ell - 1})\|^2.
    \end{align}
    This implies that
    \begin{align}
        \frac{12}{b} \mathbb{E} \sum_{j \in J_k} \left\| \nabla f_j(\varphi_k^j) - \nabla f_j (\xi_k^j) \right\|^2 & \le \frac{12 b}{n^2} \sum_{\ell=1}^k \left(1-\frac{b}{n}\right)^{k-\ell} \sum_{i=1}^n \| \nabla f_i(x_{\ell}) - \nabla f_i(x_{\ell-1})\|^2 \\
        & \le \frac{12 b}{n^2} \sum_{\ell=1}^k \left(1-\frac{b}{2 n}\right)^{k-\ell} \sum_{i=1}^n \| \nabla f_i(x_{\ell}) - \nabla f_i(x_{\ell-1})\|^2.
    \end{align}
    We include the inequality of the second line to simplify later arguments. This completes our bound for the first term of \eqref{eq:vbound}. For the second term, we recall Lemma \ref{lem:sagaforsarge}.
    \begin{equation}
        \mathbb{E} \left\|\widetilde{\nabla}^{\xi\textnormal{\tiny-SAGA}}_{k+1} - \nabla f(x_k) \right\|^2 \le \frac{3}{b^2} \sum_{\ell = 1}^k \left(1 - \frac{b}{2 n} \right)^{k - \ell} \sum_{i=1}^n \mathbb{E} \|\nabla f_i(x_{\ell}) - \nabla f_i(x_{\ell-1}) \|^2.
    \end{equation}
    Combining all of these bounds, we have shown
    \begin{align}
        & \mathbb{E} \| \widetilde{\nabla}^{\textnormal{\tiny SARGE}}_{k+1} - \widetilde{\nabla}^{\textnormal{\tiny SARGE}}_k\|^2 \notag \\
        %%%%%%%%%%%%%%%%%%%%%%%%%%%%%%%%%%%%%%%%%%%%%
        \le & \frac{12}{n} \sum_{i=1}^n \mathbb{E} \|\nabla f_i(x_{k+1}) - \nabla f_i(x_k) \|^2 + \frac{27 + 12 b}{n^2} \sum_{\ell=1}^k \left( 1 - \frac{b}{2 n} \right)^{k - \ell} \sum_{i=1}^n \|\nabla f_i(x_\ell) - \nabla f_i(x_{\ell - 1}) \|^2 \notag \\
        & + \frac{3 b^2}{2 n^2} \left\|\nabla f(x_k) - \widetilde{\nabla}^{\textnormal{\tiny SARGE}}_k \right\|^2.
    \end{align}
    
\end{proof}
Lemma \ref{lem:vbound} allows us to take advantage of the recursive structure of our gradient estimate. With this lemma established, we can prove a bound on the MSE.

\begin{lemma}
\label{lem:sargerecurse}
    The SARGE gradient estimator satisfies the following recursive bound:
\begin{align}
    & \mathbb{E} \|\widetilde{\nabla}^{\textnormal{\tiny SARGE}}_{k+1} - \nabla f(x_{k+1})\|^2 \notag \\
    \le & \left( 1 - \frac{b}{n} + \frac{3 b^2}{2 n^2} \right) \mathbb{E} \left\| \widetilde{\nabla}^{\textnormal{\tiny SARGE}}_k - \nabla f(x_k) \right\|^2 + \frac{12}{n} \sum_{i=1}^n \mathbb{E} \|\nabla f_i(x_{k+1}) - \nabla f_i(x_k) \|^2 \\
    & + \frac{27 + 12 b}{n^2} \sum_{\ell=1}^k \left( 1 - \frac{b}{2 n} \right)^{k - \ell} \sum_{i=1}^n \mathbb{E} \|\nabla f_i(x_{\ell}) - \nabla f_i(x_{\ell - 1}) \|^2. 
\end{align}
\end{lemma}

\begin{proof}
 
The beginning of our proof is similar to the proof of the variance bound for the SARAH gradient estimator in \cite[Lem. 2]{sarah}.
\begin{align}
    & \E \|\widetilde{\nabla}^{\textnormal{\tiny SARGE}}_{k+1} - \nabla f(x_{k+1})\|^2 \\
    = & \E \left\| \widetilde{\nabla}^{\textnormal{\tiny SARGE}}_k - \nabla f(x_k) + \nabla f(x_k) - \nabla f(x_{k+1}) + \widetilde{\nabla}^{\textnormal{\tiny SARGE}}_{k+1} - \widetilde{\nabla}^{\textnormal{\tiny SARGE}}_k \right\|^2 \\
    = & \left\| \widetilde{\nabla}^{\textnormal{\tiny SARGE}}_k - \nabla f(x_k) \right\|^2 + \left\|\nabla f(x_k) - \nabla f(x_{k+1}) \right\|^2 + \E \left\| \widetilde{\nabla}^{\textnormal{\tiny SARGE}}_{k+1} - \widetilde{\nabla}^{\textnormal{\tiny SARGE}}_k \right\|^2 \notag \\
    & + 2 \langle \nabla f(x_k ) - \widetilde{\nabla}^{\textnormal{\tiny SARGE}}_k, \nabla f(x_{k+1}) - \nabla f(x_k) \rangle \\
    & \quad - 2 \left\langle \nabla f(x_k ) - \widetilde{\nabla}^{\textnormal{\tiny SARGE}}_k, \E \left[ \widetilde{\nabla}^{\textnormal{\tiny SARGE}}_{k+1} - \widetilde{\nabla}^{\textnormal{\tiny SARGE}}_k \right] \right\rangle \\
    & \quad \quad - 2 \left\langle \nabla f(x_{k+1}) - \nabla f(x_k), \E \left[ \widetilde{\nabla}^{\textnormal{\tiny SARGE}}_{k+1} - \widetilde{\nabla}^{\textnormal{\tiny SARGE}}_k \right] \right\rangle.
\end{align}
We consider each inner product separately. The first inner product is equal to
\begin{align}
    2 \langle \nabla f(x_k ) - \widetilde{\nabla}^{\textnormal{\tiny SARGE}}_k, \nabla f(x_{k+1}) - \nabla f(x_k) \rangle = & - \|\nabla f(x_k ) - \widetilde{\nabla}^{\textnormal{\tiny SARGE}}_k\|^2 - \|\nabla f(x_{k+1}) - \nabla f(x_k)\|^2 \notag \\
    & + \|\nabla f(x_{k+1}) - \widetilde{\nabla}^{\textnormal{\tiny SARGE}}_k\|^2.
\end{align}
For the next two inner products, we use the fact that
\begin{align}
    & \E \left[ \widetilde{\nabla}^{\textnormal{\tiny SARGE}}_{k+1} - \widetilde{\nabla}^{\textnormal{\tiny SARGE}}_k \right] \\
    = & \E \Bigg[ \widetilde{\nabla}^{\textnormal{\tiny SAGA}}_{k+1} - \left(1 - \frac{b}{n} \right) \widetilde{\nabla}^{\xi\textnormal{\tiny-SAGA}}_{k+1} + \left(1 - \frac{b}{n} \right) \widetilde{\nabla}^{\textnormal{\tiny SARGE}}_k \Bigg] - \widetilde{\nabla}^{\textnormal{\tiny SARGE}}_k \\
    = & \nabla f(x_{k+1}) - \left( 1 - \frac{b}{n} \right) \nabla f(x_k) - \frac{b}{n} \widetilde{\nabla}^{\textnormal{\tiny SARGE}}_k \\
    = & \nabla f(x_{k+1}) - \nabla f(x_k) + \frac{b}{n} \left( \nabla f(x_k) - \widetilde{\nabla}^{\textnormal{\tiny SARGE}}_k \right).
\end{align}
With this equality established, we see that the second inner product is equal to
\begin{align}
    & - 2 \left\langle \nabla f(x_k) - \widetilde{\nabla}^{\textnormal{\tiny SARGE}}_k, \E \left[ \widetilde{\nabla}^{\textnormal{\tiny SARGE}}_{k+1} - \widetilde{\nabla}^{\textnormal{\tiny SARGE}}_k \right] \right\rangle \\
    = & - 2 \langle \nabla f(x_k) - \widetilde{\nabla}^{\textnormal{\tiny SARGE}}_k, \nabla f(x_{k+1}) - \nabla f(x_k) \rangle - \frac{2 b}{n} \langle \nabla f(x_k ) - \widetilde{\nabla}^{\textnormal{\tiny SARGE}}_k, \nabla f(x_k) - \widetilde{\nabla}^{\textnormal{\tiny SARGE}}_k \rangle \\
    =& \| \nabla f(x_k ) - \widetilde{\nabla}^{\textnormal{\tiny SARGE}}_k \|^2 + \| \nabla f(x_{k+1}) - \nabla f(x_k) \|^2 - \|\nabla f(x_{k+1}) - \widetilde{\nabla}^{\textnormal{\tiny SARGE}}_k \|^2 - \frac{2 b}{n} \| \nabla f(x_k ) - \widetilde{\nabla}^{\textnormal{\tiny SARGE}}_k \|^2 \\
    =& \left(1 - \frac{2 b}{n} \right) \| \nabla f(x_k ) - \widetilde{\nabla}^{\textnormal{\tiny SARGE}}_k \|^2 + \| \nabla f(x_{k+1}) - \nabla f(x_k) \|^2 - \|\nabla f(x_{k+1}) - \widetilde{\nabla}^{\textnormal{\tiny SARGE}}_k \|^2.
\end{align}
The third inner product can be bounded using a similar procedure.
\begin{align}
    & - 2 \left\langle \nabla f(x_{k+1}) - \nabla f(x_k), \E \left[ \widetilde{\nabla}^{\textnormal{\tiny SARGE}}_{k+1} - \widetilde{\nabla}^{\textnormal{\tiny SARGE}}_k \right] \right\rangle \\
    = & - 2 \langle \nabla f(x_{k+1}) - \nabla f(x_k), \nabla f(x_{k+1}) - \nabla f(x_k) \rangle - \frac{2 b}{n} \langle \nabla f(x_{k+1} ) - \nabla f(x_k), \nabla f(x_k) - \widetilde{\nabla}^{\textnormal{\tiny SARGE}}_k \rangle \\
    \le & - 2 \| \nabla f(x_{k+1}) - \nabla f(x_k) \|^2 + \frac{b}{n} \|\nabla f(x_{k+1} ) - \nabla f(x_k)\|^2 + \frac{b}{n} \|\nabla f(x_k) - \widetilde{\nabla}^{\textnormal{\tiny SARGE}}_k \|^2 \\
    = & - \left(2 - \frac{b}{n}\right) \| \nabla f(x_{k+1}) - \nabla f(x_k) \|^2 + \frac{1}{n} \|\nabla f(x_k) - \widetilde{\nabla}^{\textnormal{\tiny SARGE}}_k \|^2,
\end{align}
where the inequality is Young's. Altogether and after applying the full expectation operator, we have
\begin{align}
    & \mathbb{E} \|\widetilde{\nabla}^{\textnormal{\tiny SARGE}}_{k+1} - \nabla f(x_{k+1})\|^2 \notag \\
    \le & \left( 1 - \frac{b}{n} \right) \mathbb{E} \left\| \widetilde{\nabla}^{\textnormal{\tiny SARGE}}_k - \nabla f(x_k) \right\|^2 - \left( 1 - \frac{b}{n} \right) \mathbb{E} \|\nabla f(x_{k+1}) - \nabla f(x_k)\|^2 + \mathbb{E} \| \widetilde{\nabla}^{\textnormal{\tiny SARGE}}_{k+1} - \widetilde{\nabla}^{\textnormal{\tiny SARGE}}_k\|^2 \\
    \le & \left( 1 - \frac{b}{n} \right) \mathbb{E} \left\| \widetilde{\nabla}^{\textnormal{\tiny SARGE}}_k - \nabla f(x_k) \right\|^2 + \mathbb{E} \| \widetilde{\nabla}^{\textnormal{\tiny SARGE}}_{k+1} - \widetilde{\nabla}^{\textnormal{\tiny SARGE}}_k\|^2.
\end{align}
Finally, we bound the last term on the right using Lemma \ref{lem:vbound}.
\begin{align}
    & \mathbb{E} \|\widetilde{\nabla}^{\textnormal{\tiny SARGE}}_{k+1} - \nabla f(x_{k+1})\|^2 \notag \\
    \le & \left( 1 - \frac{b}{n} + \frac{3 b^2}{2 n^2} \right) \mathbb{E} \left\| \widetilde{\nabla}^{\textnormal{\tiny SARGE}}_k - \nabla f(x_k) \right\|^2 + \frac{12}{n} \sum_{i=1}^n \mathbb{E} \|\nabla f_i(x_{k+1}) - \nabla f_i(x_k) \|^2 \notag \\
    & \quad \quad + \frac{27 + 12 b}{n^2} \sum_{\ell=1}^k \left( 1 - \frac{b}{2 n} \right)^{k - \ell} \sum_{i=1}^n \mathbb{E} \|\nabla f_i(x_{\ell}) - \nabla f_i(x_{\ell - 1}) \|^2. 
\end{align}
\end{proof}

Lemma \ref{lem:sargerecurse} shows that the SARGE gradient estimator satisfies the MSEB property with suitably chosen parameters.

\begin{corollary}
\label{cor:sargemseb}
    The SARGE gradient estimator with $b \le n / 3$ satisfies the MSEB property with $M_1 = 12$, $M_2 = (27 + 12 b) / n^2$, $\rho_M = \frac{b}{2 n}$, $\rho_B = b / n$, and $\rho_F = \frac{b}{2 n}$.
\end{corollary}

\begin{proof}
    It is easy to see that $\rho_B = b / n$ by computing the expectation of the SARGE gradient estimator.
    \begin{align}
        \nabla f(x_{k+1}) - \E \widetilde{\nabla}^{\textnormal{\tiny SARGE}}_{k+1} & = \nabla f(x_{k+1}) - \E \left[ \widetilde{\nabla}^{\textnormal{\tiny SAGA}}_{k+1} - \left(1 - \frac{b}{n}\right) \left( \widetilde{\nabla}^{\xi\textnormal{\tiny-SAGA}}_{k+1} - \widetilde{\nabla}^{\textnormal{\tiny SARGE}}_k \right) \right] \\
        & = \left(1-\frac{b}{n}\right) \left( \nabla f(x_k) - \widetilde{\nabla}^{\textnormal{\tiny SARGE}}_k \right).
    \end{align}
    The result of Lemma \ref{lem:sargerecurse} makes it clear that $M_1 = 12$. To determine $\rho_M$, we must first choose a suitable sequence $\mathcal{M}_k$. Let $\mathcal{M}_k = \mathbb{E} \|\widetilde{\nabla}^{\textnormal{\tiny SARGE}}_{k+1} - \nabla f(x_{k+1})\|^2$. The requirement that $b \le n / 3$ implies $1 - \frac{b}{n} + \frac{3 b^2}{2 n^2} \le 1 - \frac{b}{2 n}$, so Lemma \ref{lem:sargerecurse} ensures that with $\rho_M =  \frac{b}{2 n}$, $\mathcal{M}_k \le (1-\rho_M) \mathcal{M}_{k-1}$.
    
    Finally, we must compute $M_2$ and $\rho_F$ with respect to some sequence $\mathcal{F}_k$. Lemma \ref{lem:sargerecurse} motivates the choice
    \begin{equation}
        \mathcal{F}_k = \sum_{\ell=1}^k \left( 1 - \frac{b}{2 n} \right)^{k - \ell} \sum_{i=1}^n \mathbb{E} \|\nabla f_i(x_{\ell}) - \nabla f_i(x_{\ell - 1}) \|^2,
    \end{equation}
    and the choices $M_2 = \frac{27 + 12 b}{n^2}$ and $\rho_F = \frac{b}{2 n}$ are clear.
\end{proof}

To prove the convergence rates of Theorem \ref{thm:sarge}, we simply combine the MSEB constants of Corollary \ref{cor:sargemseb} with Theorems \ref{thm:main1} and \ref{thm:main2}.

\end{appendix}

\end{document}